\documentclass[11pt,reqno]{amsart}
\usepackage{mathrsfs}
\usepackage{cite}

\def\subjclass#1{{\renewcommand{\thefootnote}{}%
\footnote{\emph{Mathematics Subject Classification (2020):} #1}}}

\usepackage{amsfonts}
\usepackage{amssymb}

\DeclareMathOperator{\curl}{curl}
\DeclareMathOperator{\divg}{div}

\date{\today}

\hoffset=- 2cm \voffset=- 0cm 
 \theoremstyle{plain}
\newtheorem{Thm}{Theorem}

\newtheorem{Rem}[Thm]{Remark}
\newtheorem{Lem}[Thm]{Lemma}
\newtheorem{Cor}[Thm]{Corollary}

\usepackage{amssymb}
\usepackage{color}

\def\0{\mathbf 0}

\def\v{\vskip}

\def\div{\text{\rm div\,}}
\def\curl{\text{\rm curl\,}}

\errorcontextlines=0 \numberwithin{equation}{section}
\numberwithin{Thm}{section}
\allowdisplaybreaks

\begin{document}
\large

\title[Liouville type theorems]
{Some new Liouville type theorems for the 3D stationary magneto-micropolar fluid equations}


\author[]{Zhibing Zhang}
\address{Zhibing Zhang: School of Mathematics and Physics, Key Laboratory of Modeling, Simulation and Control of Complex Ecosystem in Dabie Mountains of Anhui Higher Education Institutes, Anqing Normal University, Anqing 246133, People's Republic of China}
\email{zhibingzhang29@126.com}%

\author[]{Qian Zu}
\address{Qian Zu: School of Mathematics and Statistics, Anhui Normal University, Wuhu 241002, People's Republic of China}
\email{qianzu5822@163.com}

\thanks{}



\begin{abstract}
In this paper, we investigate Liouville type theorems for the 3D stationary magneto-micropolar fluid equations and micropolar fluid equations. Adopting an iteration procedure, taking advantage of the special structure of the equations and using a novel combination of interpolation techniques, we establish Liouville type theorems if the smooth solution satisfies certain growth conditions in terms of $L^p$-norms on the annuli. Furthermore, combining the energy method and some subtle ODE analysis, we relax the growth conditions on the velocity field and the magnetic field by logarithmic factors and obtain logarithmic improvement of Liouville type theorems. Compared with the velocity and the magnetic field, we raise the most relaxed restriction for the angular velocity. More specifically, we allow $L^q$-norm of the angular velocity on the annuli to grow polynomially at any degree, i.e. $\|\omega\|_{L^q(B_{2R}\backslash B_{3R/2})}$ is permitted to grow as fast as $R^N$ at infinity, where $N$ is an arbitrary positive integer.
\end{abstract}
\maketitle

\section{Introduction}
Consider the stationary magneto-micropolar fluid equations in $\mathbb{R}^{3}$:
\begin{align}\label{equ1.1}
  \left\{
    \begin{array}{ll}
     -\Delta u+(u\cdot\nabla) u+\nabla \pi=\chi\curl \omega+(b\cdot\nabla) b,  \\
    -\Delta\omega+(u\cdot\nabla)\omega-\nabla(\mathrm{div}\omega)+2\chi\omega=\chi\curl u,  \\
   -\Delta b+(u\cdot\nabla) b-(b\cdot\nabla) u=0, \\
     \divg u=\divg b=0,
\end{array}
  \right.
\end{align}
 where $u$ is the velocity field, $\omega$ is the angular velocity, $b$ is the magnetic field,  $\pi$ is the scalar pressure, $\curl u=\nabla\times u$. The positive parameter $\chi$ represents the micro-rotational viscosity. When $b=0$, \eqref{equ1.1} reduces to the micropolar fluid equations. When $\omega=b=0$ and $\chi=0$, \eqref{equ1.1} becomes the Navier-Stokes equations. Since the magneto-micropolar fluid equations have important physical background, rich phenomena, mathematical complexity and challenges, it has attracted the attention of many physicists and mathematicians, and many interesting results have been established, see \cite{AS74,BRF03,LLZ24,R97,RB88,TWZ19,ZWX25,Zhao23} for example.

In recent years, there have been many efforts to establish Liouville type theorems for various fluid equations. For the three-dimensional steady incompressible Navier-Stokes equations, many scholars proved the triviality of $u$ under some additional conditions. For example, Galdi \cite{Galdi} proved a Liouville type theorem, assuming that the velocity field $u \in L^{\frac{9}{2}}(\mathbb{R}^{3})$. As an extension, Chamorro et al. \cite{DJL21} showed that the condition  $u\in L^p\left(\mathbb{R}^{3}\right)$ with $3\leq p\leq \frac{9}{2}$ is sufficient to get Liouville type result. Seregin \cite{Seregin18} handled the pressure term by applying the Bogovskii map on the ball, and established the Liouville type theorem if $u$ fulfils
$$\displaystyle\sup_{R>0}R^{\alpha}\left(\frac{1}{|B_{R}|}\int_{B_{R}}|u|^{p}dx\right)^{\frac{1}{p}}<+\infty, ~~~~~~~~~~~~with~~~~~~~~~~~~\frac{3}{2}<p<3~~~~~~~~~~~~and~~~~~~~~~~~~\alpha>\frac{6p-3}{8p-6}.$$
Recently, by using an iteration argument,  Cho et al. \cite{CNY24} got an improvement of Seregin's result. More precisely, they showed Liouville type theorem holds if $u$ satisfies one of the following conditions
\begin{align*}
\aligned
&(\mathrm{i})\liminf_{R\rightarrow+\infty}R^{-(\frac{2}{p}-\frac{1}{3})}\|u\|_{L^p(B_{2R}\backslash B_R)}<+\infty,\;\frac{3}{2}<p<3,\\
&(\mathrm{ii})\liminf_{R\rightarrow+\infty}R^{-\frac{1}{3}}\|u\|_{L^3(B_{2R}\backslash B_R)}=0.
\endaligned
\end{align*}
Very recently, Cho-Yang \cite{CY25} got a logarithmic improvement of Liouville type results for the steady Navier-Stokes equations. Specifically, they proved the triviality of the smooth solution of the Navier-Stokes equations if
$$
\limsup_{R\rightarrow+\infty}\frac{\|u\|_{L^p(B_{2R}\backslash B_\frac{3R}{2})}}{R^{\frac{2}{p}-\frac{1}{3}}(\ln R)^{\frac{3}{p}-1}}<+\infty,\;\frac{3}{2}<p<3.
$$
They proved this result by a case-by-case discussion based on
\begin{align*}
\liminf_{R\rightarrow+\infty}\frac{\|u\|_{L^3(B_{2R}\backslash B_\frac{3R}{2})}}{R^{\frac{2}{3}-\frac{1}{p}}(\ln R)^{\frac{3}{p}-1}}<+\infty\text{ or }=+\infty.
\end{align*}
To the best of our knowledge, their method can not be applied to obtain Liouville type results of the logarithmic version for the systems coupled with the Navier-Stokes equations.
For more interesting works on the Liouville type results for the Navier-Stokes equations, one can refer to \cite{Chae14,CW19,KNSS09,KTW17,Seregin16,SW19,Tsai21}.

 So far, there are few Liouville type results for the micropolar fluid equations and the magneto-micropolar fluid equations. Kim and Ko \cite{KK} proved the vanishing property of a smooth solution $(u,\omega,b)$  to \eqref{equ1.1} with $\chi=1$ under the following assumption
\begin{align*}
(u,\omega,b)\in L^p\left(\mathbb{R}^{3}\right)\text{ with }2\leq p< \frac{9}{2}.
\end{align*}
Lately, Cho et al. \cite[Theorem 2]{CNY25} obtained the Liouville type theorem for \eqref{equ1.1} with $\chi=1$ provided that
\begin{align}\label{later1.2}
\liminf\limits_{R\rightarrow+\infty}\left(R^{-(\frac{2}{p}-\frac{1}{3})}\|u\|_{L^p(B_{2R}\backslash B_R)}
+R^{-h(q)}\|\omega\|_{L^q(B_{2R}\backslash B_R)}+R^{-(\frac{2}{r}-\frac{1}{3})}\|b\|_{L^r(B_{2R}\backslash B_R)}\right)<+\infty,
\end{align}
where $p,q,r$ satisfy $\frac{3}{2}< p<3$, $1\leq q\leq +\infty$, $1\leq r\leq 3$, and $h(q)$ is defined by
$h(q):=\frac{2}{q}-\frac{1}{3}$ for $1\leq q\leq 3$, $h(q):=1-\frac{2}{q}$ for $3<q\leq+\infty$. They also obtained the Liouville type theorem for \eqref{equ1.1} under
\begin{align}\label{later1.3}
\liminf\limits_{R\rightarrow+\infty}\left(R^{-\frac{1}{3}}\|u\|_{L^3(B_{2R}\backslash B_R)}
+R^{-h(q)}\|\omega\|_{L^q(B_{2R}\backslash B_R)}+R^{-(\frac{2}{r}-\frac{1}{3})}\|b\|_{L^r(B_{2R}\backslash B_R)}\right)=0,
\end{align}
where $q,r$ satisfy $1\leq q\leq +\infty$, $1\leq r\leq 3$. As a corollary (see \cite[Theorem 3]{CNY25}), they showed that if
\begin{align}\label{later1.4}
u\in L^p(\mathbb{R}^3),\;\omega\in L^q(\mathbb{R}^3),\;b\in L^r(\mathbb{R}^3)\text{ with } \frac{3}{2}<p\leq\frac{9}{2},\;1\leq q\leq+\infty,\;1\leq r\leq \frac{9}{2},
\end{align}
then $u=\omega=b=0$.
On the other hand, Chamorro et al. \cite{CLV} studied the following micropolar fluid equations with $\kappa\gg1$:
\begin{align}\label{equ1.3}
  \left\{
    \begin{array}{ll}
     -\Delta u+(u\cdot\nabla) u+\nabla \pi=\frac{1}{2}\curl \omega,  \\
    -\Delta\omega+(u\cdot\nabla)\omega-\nabla(\mathrm{div}\omega)+\kappa\omega=\frac{1}{2}\curl u,  \\
     \divg   u=0,
\end{array}
  \right.
\end{align}
and they proved $(u,\omega)=0$ if
\begin{align*}
u\in \dot{H}^1(\mathbb{R}^{3})\cap L^p(\mathbb{R}^{3}),\omega\in{H}^1(\mathbb{R}^{3}),\pi\in \dot{H}^{\frac{1}{2}}(\mathbb{R}^{3}),\text{ where }3\leq p\leq \frac{9}{2}.
\end{align*}
As mentioned in \cite{CLV}, $\kappa\gg1$ is a technical parameter, which is due to the limitation of the estimation methodology.

Our goal of this paper is to establish Liouville type results for \eqref{equ1.1} and \eqref{equ1.3}, provided that $L^p$-norms of the velocity field $u$, the angular velocity $\omega$ and the magnetic field $b$ on the annuli $B_{2R}\backslash B_\frac{3R}{2}$ satisfy some growth conditions at infinity. In the framework of an iteration argument, we make full use of the special structure of the equations and utilize a novel combination of several different interpolation inequalities (see the proofs of the key Lemmas \ref{Lemma2.4} and \ref{Lemma2.6}) to achieve our goal. Furthermore, inspired by the works of \cite{CY25,ZZB25}, combining the energy method and some subtle ODE analysis, we succeed in relaxing the growth conditions on the velocity field and the magnetic field by logarithmic factors and obtain logarithmic improvement version of Liouville type theorems. Compared with the velocity field and the magnetic field, we raise the most relaxed restriction for the angular velocity. More precisely, we allow $L^q$-norm of the angular velocity on the annuli $B_{2R}\backslash B_\frac{3R}{2}$ to grow polynomially at any degree. Our results significantly improve the results of Cho et al. \cite{CNY25}, see
Remark \ref{Rem1.4} for details. Besides, the requirement on the technical parameter  $\kappa$ in \eqref{equ1.3} can be weakened to $\kappa>\frac{1}{4}$.

In order to show our Liouville type results, we need to introduce some notations. Let $p'$ denote the conjugate exponent to $p$, i.e., $p'=\frac{p}{p-1}$. For simplicity, we denote
$$X_{p,\alpha}(R)=R^{-\alpha}\|u\|_{L^p\left(A_R\right)},\;Y_{q,\beta}(R)=R^{-\beta}\|\omega\|_{L^q\left(A_R\right)},\;Z_{r,\gamma}(R)=R^{-\gamma}\|b\|_{L^r\left(A_R\right)},$$
$$X_{p,\alpha,\lambda}(R)=R^{-\alpha}(\ln R)^{-\lambda}\|u\|_{L^p\left(A_R\right)},\;Z_{r,\gamma,\nu}(R)=R^{-\beta}(\ln R)^{-\nu}\|b\|_{L^q\left(A_R\right)},$$
where $A_R=B_{2R}\backslash \overline{B_\frac{3R}{2}}$. To begin with, we state Liouville type theorems for the magneto-micropolar fluid equations.

\begin{Thm}\label{main1}
Let $(u,\pi,\omega,b)$ be a smooth solution of \eqref{equ1.1}. Suppose $q\in[1,+\infty]$, $r\in [1,6]$, $\beta\in\left[0,+\infty\right)$, $\gamma\in\left[0,\frac{3}{r}-\frac{1}{2}\right]$, $\chi\in(0,2)$. Suppose that one of the following assumptions holds
\begin{align*}
\mathrm{(A1)}\;&\liminf\limits_{R\rightarrow+\infty}\left[X_{p,\alpha}(R)+Y_{q,\beta}(R)+Z_{r,\gamma}(R)\right]<+\infty,\text{ where  $p,r,\alpha,\gamma$ satisfy}\\
&p\in\left(\frac{3}{2},3\right),\;r\in[1,2p'),\;\alpha\in\left[0,\frac{2}{p}-\frac{1}{3}\right],\;\alpha+\frac{(4p-6)r}{(6-r)p}\gamma\leq1;\\
\mathrm{(A2)}\;&\liminf\limits_{R\rightarrow+\infty}\left[X_{p,\alpha}(R)+Y_{q,\beta}(R)+Z_{r,\gamma}(R)\right]<+\infty,\text{ where  $p,r,\alpha,\gamma$ satisfy}\\
&p\in\left(\frac{3}{2},3\right),\;r\in[2p',6],\;\alpha\in\left[0,\frac{2}{p}-\frac{1}{3}\right],\;\alpha+2\gamma<\frac{3}{p}+\frac{6}{r}-2;\\
\mathrm{(A3)}\;&\liminf\limits_{R\rightarrow+\infty}X_{p,\alpha}(R)=0,\;\limsup\limits_{R\rightarrow+\infty}\left[Y_{q,\beta}(R)+Z_{r,\gamma}(R)\right]<+\infty,\text{ where  $p,r,\alpha,\gamma$ satisfy}\\
&p\in\left[3,\frac{9}{2}\right],\;r\in[1,2p'),\;\alpha\in\left[0,\frac{3}{p}-\frac{2}{3}\right],\;\alpha+\frac{(4p-6)r}{(6-r)p}\gamma\leq1;\\
\mathrm{(A4)}\;&\liminf\limits_{R\rightarrow+\infty}X_{p,\alpha}(R)=0,\;\limsup\limits_{R\rightarrow+\infty}\left[Y_{q,\beta}(R)+Z_{r,\gamma}(R)\right]<+\infty,\text{ where  $p,r,\alpha,\gamma$ satisfy}\\
&p\in\left[3,\frac{9}{2}\right],\;r\in[2p',6],\;\alpha\in\left[0,\frac{3}{p}-\frac{2}{3}\right],\;\alpha+2\gamma\leq\frac{3}{p}+\frac{6}{r}-2.
\end{align*}
Then $u=\omega=b=0$.
\end{Thm}

Roughly speaking, we can relax the growth conditions in Theorem \ref{main1} by logarithmic factors in certain situations, where at least one of the inequalities $p<3$ and $r<2p'$ is satisfied. For the convenience of presenting our Liouville type results of logarithmic improvement version, we make two basic assumptions on the parameters $\lambda$ and $\nu$. The first one is
\begin{align}\label{ass1.8}
\lambda\in\left[0,\frac{3}{p}-1\right] \text{ for }p\in\left(\frac{3}{2},3\right)\text{ and }\lambda=0\text{ for }p\in\left[3,\frac{9}{2}\right],\;\nu\geq0.
\end{align}
The second one is
\begin{subequations}
\begin{align}
\text{when }&r\in[1,2p')\text{ and }(\alpha,\gamma)\neq\left(\frac{3}{p}-1,\frac{3}{r}-\frac{1}{2}\right),\text{ we assume}\notag\\
&\lambda+\frac{(4p-6)r}{(6-r)p}\nu\leq\frac{6p-3r(p-1)}{(6-r)p};\label{ass1.9a}\\
\text{when }&r\in[1,2p')\text{ and }(\alpha,\gamma)=\left(\frac{3}{p}-1,\frac{3}{r}-\frac{1}{2}\right),\text{ we assume}\notag\\
&\lambda+\frac{6p+pr-3r}{(6-r)p}\nu\leq\frac{6p-3r(p-1)}{2(6-r)p}.\label{ass1.9b}
\end{align}
\end{subequations}

Now we are ready to show our Liouville type results of logarithmic improvement version.
\begin{Thm}\label{main2}
Let $(u,\pi,\omega,b)$ be a smooth solution of \eqref{equ1.1}. Suppose $q\in[1,+\infty]$, $r\in [1,6]$, $\beta\in\left[0,+\infty\right)$, $\gamma\in\left[0,\frac{3}{r}-\frac{1}{2}\right]$, $\chi\in(0,2)$. Assume that $\lambda,\nu$ satisfy \eqref{ass1.8} and \eqref{ass1.9a}-\eqref{ass1.9b}.  Suppose that one of the following assumptions holds
\begin{align*}
\mathrm{(B1)}\;&\limsup\limits_{R\rightarrow+\infty}\left[X_{p,\alpha,\lambda}(R)+Y_{q,\beta}(R)+Z_{r,\gamma,\nu}(R)\right]<+\infty,\text{ where  $p,r,\alpha,\gamma$ satisfy}\\
&p\in\left(\frac{3}{2},3\right),\;r\in[1,2p'),\;\alpha\in\left[0,\frac{2}{p}-\frac{1}{3}\right],\;\alpha+\frac{(4p-6)r}{(6-r)p}\gamma\leq1,\\
\mathrm{(B2)}\;&\limsup\limits_{R\rightarrow+\infty}\left[X_{p,\alpha,\lambda}(R)+Y_{q,\beta}(R)+Z_{r,\gamma,\nu}(R)\right]<+\infty,\text{ where  $p,r,\alpha,\gamma$  satisfy}\\
&p\in\left(\frac{3}{2},3\right),\;r\in[2p',6],\;\alpha\in\left[0,\frac{2}{p}-\frac{1}{3}\right],\;\alpha+2\gamma<\frac{3}{p}+\frac{6}{r}-2;\\
\mathrm{(B3)}\;&\lim\limits_{R\rightarrow+\infty}X_{p,\alpha}(R)=0,\;\limsup\limits_{R\rightarrow+\infty}\left[Y_{q,\beta}(R)+Z_{r,\gamma,\nu}(R)\right]<+\infty,\text{ where  $p,r,\alpha,\gamma$  satisfy}\\
&p\in\left[3,\frac{9}{2}\right],\;r\in[1,2p'),\;\alpha\in\left[0,\frac{3}{p}-\frac{2}{3}\right],\;\alpha+\frac{(4p-6)r}{(6-r)p}\gamma\leq1.
\end{align*}
Then $u=\omega=b=0$.
\end{Thm}

As a consequence of Theorem \ref{main1} and Remark \ref{Rem3.1} in Section \ref{sec3}, we obtain the following Liouville type theorem for \eqref{equ1.1} in Lebesgue spaces.
\begin{Cor}\label{Cor1.3}
Let $(u,\pi,\omega,b)$ be a smooth solution of \eqref{equ1.1} and $\chi\in(0,2)$. Suppose $u\in L^p(\mathbb{R}^3)$, $\omega\in L^q(\mathbb{R}^3)$, $b\in L^r(\mathbb{R}^3)$. Then $u=\omega=b=0$ if one of the following assumptions holds
\begin{align*}
\mathrm{(C1)}&\;p\in\left(\frac{3}{2},3\right),\;q\in[1,+\infty],\;r\in[1,6];\\
\mathrm{(C2)}&\;p\in\left[3,\frac{9}{2}\right],\;q\in[1,+\infty],\;r\in[1,2p');\\
\mathrm{(C3)}&\;p\in\left[3,\frac{9}{2}\right],\;q\in[1,+\infty],\;r\in[2p',6],\text{ with }\;\frac{1}{p}+\frac{2}{r}\geq\frac{2}{3}.
\end{align*}
\end{Cor}
\begin{Rem}\label{Rem1.4}
Theorem 2 and Theorem 3 in \cite{CNY25} are special cases of Theorem \ref{main1} and Corollary \ref{Cor1.3} in this paper, respectively. Moreover, we make improvements mainly in three aspects:
\begin{itemize}
\item[(1)] In our setting, the parameter $\beta$ can be arbitrarily large. It is permitted that $\|\omega\|_{L^q(A_R)}$ grows as fast as $R^N$ at infinity, where $N$ is an arbitrary positive integer.
\item[(2)] For the velocity field and the magnetic field, we relax the growth conditions by logarithmic factors.
\item[(3)] We relax the range of the parameters $\alpha$ and $\gamma$, which is different from the setting in \eqref{later1.2} and \eqref{later1.3}.
This flexibility makes it possible to extend the range of the integrability index $r$ from $1\leq r\leq \frac{9}{2}$ to $1\leq r\leq 6$, see \eqref{later1.4} and the assumptions in Corollary \ref{Cor1.3}.
\end{itemize}
\end{Rem}

Similarly to the magneto-micropolar fluid equations, we can also establish Liouville type results for the micropolar fluid equations.
\begin{Thm}\label{main3}
Let $(u,\pi,\omega)$ be a smooth solution of \eqref{equ1.3}. Suppose  $q\in [1,+\infty]$, $\beta\in\left[0,+\infty\right)$ and $\kappa\in(\frac{1}{4},+\infty)$. Assume that one of the following assumptions holds
\begin{align*}
\mathrm{(D1)}\;&\liminf\limits_{R\rightarrow+\infty}\left[X_{p,\alpha}(R)+Y_{q,\beta}(R)\right]<+\infty,\text{ where }p\in\left(\frac{3}{2},3\right),\;\alpha\in\left[0,\frac{2}{p}-\frac{1}{3}\right];\\
\mathrm{(D2)}\;&\liminf\limits_{R\rightarrow+\infty}X_{p,\alpha}(R)=0,\;\limsup\limits_{R\rightarrow+\infty}Y_{q,\beta}(R)<+\infty,\text{ where }p\in\left[3,\frac{9}{2}\right],\;\alpha\in\left[0,\frac{3}{p}-\frac{2}{3}\right].
\end{align*}
Then $u =\omega= 0$.
\end{Thm}

\begin{Thm}\label{main4}
Let $(u,\pi,\omega)$ be a smooth solution of \eqref{equ1.3}. Suppose $p\in(\frac{3}{2},3)$, $q\in [1,+\infty]$, $\alpha=\frac{2}{p}-\frac{1}{3}$, $\beta\in\left[0,+\infty\right)$, $\lambda=\frac{3}{p}-1$ and $\kappa\in(\frac{1}{4},+\infty)$.
Assume that
$$\limsup\limits_{R\rightarrow+\infty}[X_{p,\alpha,\lambda}(R)+Y_{q,\beta}(R)]<+\infty.$$
Then $u =\omega= 0$.
\end{Thm}

\begin{Cor}\label{Cor1.7}
Let $(u,\pi,\omega)$ be a smooth solution of \eqref{equ1.3} and $\kappa\in(\frac{1}{4},+\infty)$. Suppose $u\in L^p(\mathbb{R}^3)$, $\omega\in L^q(\mathbb{R}^3)$, where
$p\in(\frac{3}{2},\frac{9}{2}]$, $q\in[1,+\infty]$. Then $u=\omega=0$.
\end{Cor}

The rest of this paper is organized as follows. The property of the Bogovskii map, the refined Giaquinta's iteration lemma and some key lemmas are introduced in Section \ref{sec2}. The proofs of Theorem \ref{main1} and Theorem \ref{main2} are presented in Section \ref{sec3} and Section \ref{sec4}, respectively.  It should be noted that throughout this article, we use $C$ to denote a finite inessential constant which may be different from line to line.

\v0.1in
\section{Preliminaries}\label{sec2}

To handle the pressure term, we need to introduce the Bogovskii map (see \cite[Lemma III.3.1]{Galdi} and \cite[Lemma 3]{Tsai21}).
\begin{Lem}\label{Lem2.2}
Let $R>0,1<k<\infty$ and $E=B_{kR}\setminus \overline{B_R}$ be an annulus in $\mathbb{R}^{3}$. Denote $L_{0}^\sigma(E):=\{f\in L^\sigma(E):\int_E fdx=0\}$ with $1<\sigma<\infty$. There exists a linear map
\begin{align*}
\mathrm{Bog}:L_{0}^\sigma(E)\rightarrow W_{0}^{1,\sigma}(E),
\end{align*}
such that for any  $f\in L_{0}^\sigma(E),v=\mathrm{Bog}f\in W_{0}^{1,\sigma}(E)$ is a vector field satisfying
\begin{align*}
 \hspace{0.3cm}\divg  v=f,\hspace{0.3cm}\|\nabla v\|_{L^\sigma(E)}\leq \frac{C_{\sigma}}{(k-1)k^{1-\frac{1}{\sigma}}}\|f\|_{L^\sigma(E)},
\end{align*}
where $C_\sigma$ is independent of $k$ and $R$.
\end{Lem}

Next we present the following Giaquinta's iteration lemma \cite[Lemma 3.1]{Giaquinta}, and its proof can be found in \cite[Lemma 2.1]{CL24}.
\begin{Lem}\label{Lem2.3}
Let $f(t)$ be a non-negative bounded function on $\left[r_0, r_1\right] \subset \mathbb{R}^{+}$. If there are non-negative constants $A_i, B_i,\alpha_i$, $i=1,2,\cdots,m$, and a parameter $\theta_0 \in[0,1)$ such that for any $r_0 \leq s<t \leq r_1$, it holds that
$$f(s) \leq \theta_0 f(t)+\sum_{i=1}^m\left(\frac{A_i}{(t-s)^{\alpha_i}}+B_i\right),$$
then
$$f(s) \leq C\sum_{i=1}^m\left(\frac{A_i}{(t-s)^{\alpha_i}}+B_i\right),$$
where $C$ is a constant depending on $\alpha_1,\alpha_2,\cdots,\alpha_m$ and $\theta_0$.
\end{Lem}

In order to establish the Liouville type results, we need to estimate several integrals involving $u$, $\omega$ and $b$.
Denote
\begin{align*}
&J_1=\frac{1}{(t-s)^2}\int_{B_t \backslash B_s}(|u|^2+|b|^2) d x,\;J_2=\frac{1}{(t-s)^2}\int_{B_t \backslash B_s}|\omega|^2d x,\\
&J_3=\frac{1}{t-s} \int_{B_t \backslash B_\frac{3R}{2}}|u|^3 d x,\;J_4=\frac{1}{t-s}\int_{B_t \backslash B_s}|\omega|^2|u|dx,\\
&J_5=\frac{1}{t-s}\|b\|_{L^{2p'}(B_t\backslash B_{\frac{3R}{2}})}^2\|u\|_{L^p(B_t\backslash B_{\frac{3R}{2}})}.
\end{align*}
The estimates on $J_1$, $J_2$, $J_3$, $J_4$, $J_5$ are given in the next five lemmas.

\begin{Lem}\label{Lem2.4}
Let $\sqrt{3}R\leq s<t\leq 2R$ and $p,r\geq1$. Suppose that $u,b$ are smooth vector-valued functions. Then
\begin{itemize}
\item[(i)] For any $\varepsilon,\delta>0$, there exist positive constants $C$, $C_\varepsilon$ and $C_\delta$ such that
\begin{align}\label{ine2.3}
\aligned
J_1\leq&\varepsilon\|u\|_{L^{6}(B_t\backslash B_s)}^{2}+\frac{C_\varepsilon R^{\frac{6}{p}-1}}{(t-s)^{\frac{12}{p}-2}}\|u\|_{L^p(A_R)}^2+\frac{CR^{3-\frac{6}{p}}}{(t-s)^2}\|u\|_{L^p (A_R)}^2+\\
&\delta\|b\|_{L^{6}(B_t\backslash B_s)}^{2}+\frac{C_\delta R^{\frac{6}{r}-1}}{(t-s)^{\frac{12}{r}-2}}\|b\|_{L^r(A_R)}^2+\frac{CR^{3-\frac{6}{r}}}{(t-s)^2}\|b\|_{L^{r}(A_R)}^{2}.
\endaligned
\end{align}
\item[(ii)] It holds that
\begin{align}\label{ine2.4}
\aligned
J_1\leq\frac{CR^2}{(t-s)^2}\left(\|u\|_{L^6(A_R)}^2+\|b\|_{L^6(A_R)}^2\right).
\endaligned
\end{align}
\end{itemize}
\end{Lem}
\begin{proof}
Denote
\begin{align*}
J_{11}=\frac{1}{(t-s)^2}\int_{B_t\backslash B_s}|u|^2 dx,\;J_{12}=\frac{1}{(t-s)^2}\int_{B_t\backslash B_s}|b|^2dx,
\end{align*}
then $J_1=J_{11}+J_{12}$. Applying the H\"{o}lder inequality to $J_{11}$, we have
\begin{align}\label{add-ineq2}
J_{11}\leq\frac{CR}{(t-s)^2}\|u\|_{L^3(B_t\backslash B_s)}^2.
\end{align}
When $1\leq p<3$, substituting the following interpolation inequality
$$
\|u\|_{L^3}\leq \|u\|_{L^p}^\frac{p}{6-p}\|u\|_{L^6}^\frac{6-2p}{6-p}
$$
to \eqref{add-ineq2} and using the Young inequality, we obtain
\begin{align}\label{add-ineq3}
\aligned
J_{11}&\leq\frac{CR}{(t-s)^2}\|u\|_{L^p(B_t\backslash B_s)}^\frac{2p}{6-p}\|u\|_{L^6(B_t\backslash B_s)}^\frac{12-4p}{6-p}\\
&\leq\varepsilon\|u\|_{L^6(B_t\backslash B_s)}^2+\frac{C_\varepsilon R^{\frac{6}{p}-1}}{(t-s)^{\frac{12}{p}-2}}\|u\|_{L^p(A_R)}^2.
\endaligned
\end{align}
When $p\geq3$, applying the H\"{o}lder inequality to \eqref{add-ineq2}, we find that
\begin{align}\label{add-ineq4}
J_{11}\leq\frac{CR^{3-\frac{6}{p}}}{(t-s)^2}\|u\|_{L^p (A_R)}^2.
\end{align}
Combining \eqref{add-ineq3} and \eqref{add-ineq4}, we conclude that $J_{11}$ can be controlled by
$$
J_{11}\leq \varepsilon\|u\|_{L^6(B_t\backslash B_s)}^2+\frac{C_\varepsilon R^{\frac{6}{p}-1}}{(t-s)^{\frac{12}{p}-2}}\|u\|_{L^p(A_R)}^2+\frac{CR^{3-\frac{6}{p}}}{(t-s)^2}\|u\|_{L^p (A_R)}^2.
$$
Similarly, we have
$$
J_{12}\leq \delta\|b\|_{L^6(B_t\backslash B_s)}^2+\frac{C_\delta R^{\frac{6}{r}-1}}{(t-s)^{\frac{12}{r}-2}}\|b\|_{L^r(A_R)}^2+\frac{CR^{3-\frac{6}{r}}}{(t-s)^2}\|b\|_{L^r (A_R)}^2.
$$

Applying the H\"{o}lder inequality to $J_{11}$ and $J_{12}$, we can get \eqref{ine2.4}.
\end{proof}

\begin{Lem}\label{Lemma2.4}
Let $\sqrt{3}R\leq s<t\leq 2R$. Suppose that $\omega$ is a smooth vector-valued function. Define $h_1(q):=1$ for $1\leq q<2$ and $h_1(q):=0$ for $2\leq q<+\infty$. Then
\begin{itemize}
\item[(i)] Let $1\leq q<+\infty$. For any $\varepsilon,\delta>0$, $\theta\in (0,1)$, there exist positive constants $C_{\varepsilon,\theta}$ and $C_{\varepsilon,\delta,\theta}$ such that
\begin{align}\label{later2.5}
\aligned
J_2&\leq\varepsilon\|\omega\|_{L^2(B_t\backslash B_s)}^2+\delta\|\omega\|_{L^6(B_t\backslash B_s)}^2+\left(\frac{C_{\varepsilon,\delta,\theta}h_1(q)}{(t-s)^\frac{6-q}{q(1-\theta)}}+\frac{C_{\varepsilon,\theta} R^{3-\frac{6}{q}}}{(t-s)^\frac{2}{1-\theta}}\right)\|\omega\|_{L^q(A_R)}^2.
\endaligned
\end{align}
\item[(ii)] Let $q=+\infty$. For any $\varepsilon,\delta>0$, $\theta\in (0,1)$, there exists a positive constant $C_{\varepsilon,\theta}$ such that
\begin{align}\label{later2.6}
\aligned
J_2&\leq\varepsilon\|\omega\|_{L^2(B_t\backslash B_s)}^2+\frac{C_{\varepsilon,\theta}}{(t-s)^\frac{2}{1-\theta}}R^3\|\omega\|_{L^\infty(A_R)}^2.
\endaligned
\end{align}
\item[(iii)] It holds that
\begin{align}\label{later2.7}
J_2\leq\frac{1}{(t-s)^2}\|\omega\|_{L^2(A_R)}^2.
\end{align}
\end{itemize}
\end{Lem}

\begin{proof}
By the Young inequality, we obtain
\begin{align}\label{last2.8}
\aligned
J_2&=\frac{1}{(t-s)^2}\|\omega\|_{L^2(B_t\backslash B_s)}^{2\theta}\|\omega\|_{L^2(B_t\backslash B_s)}^{2(1-\theta)}\\
&\leq\varepsilon\|\omega\|_{L^2(B_t\backslash B_s)}^2+\frac{C_{\varepsilon,\theta}}{(t-s)^\frac{2}{1-\theta}}\|\omega\|_{L^2(B_t\backslash B_s)}^2.
\endaligned
\end{align}
When $1\leq q<2$, using the interpolation inequality
$$
\|\omega\|_{L^2(B_t\backslash B_s)}\leq \|\omega\|_{L^q(B_t\backslash B_s)}^\frac{2q}{6-q}\|\omega\|_{L^6(B_t\backslash B_s)}^\frac{6-3q}{6-q}
$$
and the Young inequality, we have
\begin{align}\label{later2.8}
\aligned
J_2&\leq\varepsilon\|\omega\|_{L^2(B_t\backslash B_s)}^2+\frac{C_{\varepsilon,\theta}}{(t-s)^\frac{2}{1-\theta}}\|\omega\|_{L^q(B_t\backslash B_s)}^\frac{4q}{6-q}\|\omega\|_{L^6(B_t\backslash B_s)}^\frac{12-6q}{6-q}\\
&\leq\varepsilon\|\omega\|_{L^2(B_t\backslash B_s)}^2+\delta\|\omega\|_{L^6(B_t\backslash B_s)}^2+\frac{C_{\varepsilon,\delta,\theta}}{(t-s)^\frac{6-q}{q(1-\theta)}}\|\omega\|_{L^q(B_t\backslash B_s)}^2.
\endaligned
\end{align}
When $2\leq q<+\infty$, combining \eqref{last2.8} and the H\"{o}lder inequality, we get
\begin{align}\label{later2.9}
J_2\leq\varepsilon\|\omega\|_{L^2(B_t\backslash B_s)}^2+\frac{C_{\varepsilon,\theta}}{(t-s)^\frac{2}{1-\theta}}R^{3-\frac{6}{q}}\|\omega\|_{L^q(B_t\backslash B_s)}^2.
\end{align}
Taking \eqref{later2.8} and \eqref{later2.9} into account together, we conclude that \eqref{later2.5} holds.

Substituting the following inequality
$$
\|\omega\|_{L^q(B_t\backslash B_s)}\leq CR^\frac{3}{q}\|\omega\|_{L^\infty(A_R)}
$$
into \eqref{later2.9}, we obtain \eqref{later2.6}. It is obvious that \eqref{later2.7} holds.
\end{proof}

\begin{Lem}\label{Lem2.5}
Let $\sqrt{3}R\leq s<t\leq 2R$. Suppose that $u$ is a smooth vector-valued function. Then we have the following conclusions:
\begin{itemize}
\item[(i)]  Assume $p\in\left(\frac{3}{2},3\right)$. It holds that
\begin{align}\label{ine2.6}
J_3\leq \frac{1}{t-s}\|u\|_{L^p(B_t\backslash B_{\frac{3R}{2}})}^{\frac{3p}{6-p}}\|u\|_{L^6(B_t\backslash B_{\frac{3R}{2}})}^{\frac{18-6p}{6-p}}.
\end{align}
\item[(ii)] Assume $p\in\left(\frac{3}{2},3\right)$. For any $\varepsilon>0$, there exists a positive constant $C_\varepsilon$ such that
\begin{align}\label{ine2.7}
\aligned
J_3&\leq\varepsilon\|u\|_{L^6(B_t\backslash B_{\frac{3R}{2}})}^2+\frac{C_\varepsilon}{(t-s)^\frac{6-p}{2p-3}}\|u\|_{L^p(A_R)}^\frac{3p}{2p-3}.
\endaligned
\end{align}
\item[(iii)]  Assume $p\geq3$. It holds that
\begin{align}\label{ine2.5}
J_3\leq\frac{C}{t-s}R^{3-\frac{9}{p}}\|u\|_{L^{p}(A_R)}^3.
\end{align}
\end{itemize}
\end{Lem}

\begin{proof}
By the H\"{o}lder inequality, we obtain
\begin{align}\label{add-ineq1}
J_3=\frac{1}{t-s}\int_{B_t\backslash B_{\frac{3R}{2}}}|u|^2\cdot|u|dx\leq\frac{1}{t-s}\|u\|_{L^{2p'}(B_t\backslash B_{\frac{3R}{2}})}^2\|u\|_{L^p(B_t\backslash B_{\frac{3R}{2}})}.
\end{align}
When $p\in(\frac{3}{2},3)$, we have $p<2p'<6$. Substituting the following interpolation inequality
$$
\|u\|_{L^{2p'}}\leq \|u\|_{L^p}^\frac{2p-3}{6-p}\|u\|_{L^6}^\frac{9-3p}{6-p}
$$
into \eqref{add-ineq1}, we get \eqref{ine2.6}. Applying the Young inequality to \eqref{ine2.6}, we can derive that \eqref{ine2.7} holds.

When $p\geq3$, we have $2p'\leq p$. Applying the H\"{o}lder inequality to \eqref{add-ineq1}, we obtain
\begin{align*}
J_3\leq\frac{1}{t-s}\left(\|u\|_{L^p(B_t\backslash B_{\frac{3R}{2}})}\cdot CR^{\frac{3}{2}-\frac{9}{2p}}\right)^2\|u\|_{L^p(B_t\backslash B_{\frac{3R}{2}})}\leq\frac{C}{t-s}R^{3-\frac{9}{p}}\|u\|_{L^{p}(A_R)}^3.
\end{align*}
\end{proof}

\begin{Lem}\label{Lemma2.6}
Let $\sqrt{3}R\leq s<t\leq 2R$. Suppose that $u,\omega$ are smooth vector-valued functions. Define $h_2(q):=1$ for $1\leq q<2p'$ and $h_2(q):=0$ for $2p'\leq q<+\infty$. Define $h_3(q):=1-h_2(q)$. Then
\begin{itemize}
\item[(i)] Let $\frac{3}{2}<p\leq\frac{9}{2}$, $1\leq q<+\infty$. For any $\varepsilon,\delta>0$, $\theta\in (0,1)$, there exist a positive constant $C_{\varepsilon,\delta,\theta}$ such that
\begin{align}\label{later2.13}
\aligned
J_4\leq&\varepsilon\|\omega\|_{L^2(B_t\backslash B_s)}^2+C_{\varepsilon,\delta,\theta}h_2(q)\left(\frac{1}{t-s}\|\omega\|_{L^q(A_R)}^{\frac{(6-2p')q}{(6-q)p'}(1-\theta)}\|u\|_{L^p(A_R)}\right)^{\frac{(6-q)p'}{(3-p')q(1-\theta)}}+\\
&\delta\|\omega\|_{L^6(B_t\backslash B_s)}^2+C_{\varepsilon,\delta,\theta}h_3(q)\left(\frac{1}{t-s}\|\omega\|_{L^q(A_R)}^{\frac{2q}{(q-2)p}(1-\theta)}\|u\|_{L^p(A_R)}\right)^\frac{(q-2)p}{(1-\theta)q}.
\endaligned
\end{align}
\item[(ii)] Let $\frac{3}{2}<p\leq\frac{9}{2}$, $q=+\infty$. For any $\varepsilon,\delta>0$, $\theta\in (0,1)$, there exists a positive constant $C_{\varepsilon,\delta,\theta}$ such that
\begin{align}\label{later2.14}
\aligned
J_4\leq&\varepsilon\|\omega\|_{L^2(B_t\backslash B_s)}^2+\delta\|\omega\|_{L^6(B_t\backslash B_s)}^2+C_{\varepsilon,\delta,\theta}\left(\frac{R^{\frac{3}{2p}(1-\theta)}}{t-s}\|\omega\|_{L^\infty(A_R)}^{\frac{3}{p}(1-\theta)}\|u\|_{L^p(A_R)}\right)^\frac{2p}{3(1-\theta)}.
\endaligned
\end{align}
\item[(iii)] Let $\frac{3}{2}<p\leq\frac{9}{2}$. It holds that
\begin{align}\label{ine-a3}
\aligned
J_4&\leq\frac{1}{t-s}\|u\|_{L^p(A_R)}\|\omega\|_{L^2(A_R)}^{2-\frac{3}{p}}\|\omega\|_{L^6(A_R)}^\frac{3}{p}.
\endaligned
\end{align}
\end{itemize}
\end{Lem}
\begin{proof}
By the H\"{o}lder inequality, we obtain
\begin{align}\label{later2.16}
J_4\leq\frac{1}{t-s}\|\omega\|_{L^{2p'}(B_t\backslash B_s)}^2\|u\|_{L^p(B_t\backslash B_s)}.
\end{align}
When $\frac{3}{2}<p\leq\frac{9}{2}$, $1\leq q<2p'$, using \eqref{later2.16} and the following interpolation inequalities
$$
\|\omega\|_{L^{2p'}(B_t\backslash B_s)}\leq \|\omega\|_{L^q(B_t\backslash B_s)}^\frac{(3-p')q}{(6-q)p'}\|\omega\|_{L^6(B_t\backslash B_s)}^\frac{6p'-3q}{(6-q)p'},
$$
\begin{align}\label{later2.17}
\|\omega\|_{L^{2p'}(B_t\backslash B_s)}\leq \|\omega\|_{L^2(B_t\backslash B_s)}^{1-\frac{3}{2p}}\|\omega\|_{L^6(B_t\backslash B_s)}^\frac{3}{2p},
\end{align}
we obtain
$$
\aligned
J_4&\leq\frac{1}{t-s}\|\omega\|_{L^{2p'}(B_t\backslash B_s)}^{2(1-\theta)}\|\omega\|_{L^{2p'}(B_t\backslash B_s)}^{2\theta}\|u\|_{L^p(B_t\backslash B_s)}\\
&\leq \frac{1}{t-s}\|\omega\|_{L^2(B_t\backslash B_s)}^{(2-\frac{3}{p})\theta}\|\omega\|_{L^6(B_t\backslash B_s)}^{\frac{12p'-6q}{(6-q)p'}(1-\theta)+\frac{3}{p}\theta}\|\omega\|_{L^q(B_t\backslash B_s)}^{\frac{(6-2p')q}{(6-q)p'}(1-\theta)}\|u\|_{L^p(B_t\backslash B_s)}.
\endaligned
$$
For any $a_1,a_2,a_3\geq0$, and any $\varepsilon,\delta>0$, it holds that
$$
a_1a_2a_3\leq\varepsilon a_1^{p_1}+\delta a_2^{p_2}+C_{\varepsilon,\delta} a_3^{p_3},\text{ where }p_1,p_2,p_3\in(1,+\infty)\text{ satisfy }\sum_{i=1}^3\frac{1}{p_i}=1.
$$
By the way, we point out that the above constant $C_{\varepsilon,\delta}$ actually also depends on $p_1$, $p_2$, $p_3$.
By the above Young inequality, we conclude that
\begin{align}\label{later2.18}
J_4\leq&\varepsilon\|\omega\|_{L^2(B_t\backslash B_s)}^2+\delta\|\omega\|_{L^6(B_t\backslash B_s)}^2+C_{\varepsilon,\delta,\theta}\left(\frac{1}{t-s}\|\omega\|_{L^q(B_t\backslash B_s)}^{\frac{(6-2p')q}{(6-q)p'}(1-\theta)}\|u\|_{L^p(B_t\backslash B_s)}\right)^{\frac{(6-q)p'}{(3-p')q(1-\theta)}}.
\end{align}

When $\frac{3}{2}<p\leq\frac{9}{2}$, $2p'\leq q<+\infty$, combining \eqref{later2.16}, the following interpolation inequality
$$
\|\omega\|_{L^{2p'}(B_t\backslash B_s)}\leq \|\omega\|_{L^2(B_t\backslash B_s)}^\frac{pq-2p-q}{(q-2)p}\|\omega\|_{L^q(B_t\backslash B_s)}^\frac{q}{(q-2)p}
$$
and \eqref{later2.17}, we derive that
$$
\aligned
J_4&\leq\frac{1}{t-s}\|\omega\|_{L^{2p'}(B_t\backslash B_s)}^{2(1-\theta)}\|\omega\|_{L^{2p'}(B_t\backslash B_s)}^{2\theta}\|u\|_{L^p(B_t\backslash B_s)}\\
&\leq \frac{1}{t-s}\|\omega\|_{L^q(B_t\backslash B_s)}^{\frac{2q}{(q-2)p}(1-\theta)}\|\omega\|_{L^2(B_t\backslash B_s)}^{\frac{2pq-4p-2q}{(q-2)p}(1-\theta)+(2-\frac{3}{p})\theta}\|\omega\|_{L^6(B_t\backslash B_s)}^{\frac{3}{p}\theta}\|u\|_{L^p(B_t\backslash B_s)}.
\endaligned
$$
Applying the Young inequality to the above inequality, we have
\begin{align}\label{later2.19}
\aligned
J_4\leq&\varepsilon\|\omega\|_{L^2(B_t\backslash B_s)}^2+\delta\|\omega\|_{L^6(B_t\backslash B_s)}^2+C_{\varepsilon,\delta,\theta}\left(\frac{1}{t-s}\|\omega\|_{L^q(B_t\backslash B_s)}^{\frac{2q}{(q-2)p}(1-\theta)}\|u\|_{L^p(B_t\backslash B_s)}\right)^\frac{(q-2)p}{(1-\theta)q}.
\endaligned
\end{align}
Taking \eqref{later2.18} and \eqref{later2.19} into account together, we conclude that \eqref{later2.13} holds.

Setting $q=6$ in \eqref{later2.19}, we get
\begin{align}\label{later2.20}
\aligned
J_4\leq&\varepsilon\|\omega\|_{L^2(B_t\backslash B_s)}^2+\delta\|\omega\|_{L^6(B_t\backslash B_s)}^2+C_{\varepsilon,\delta,\theta}\left(\frac{1}{t-s}\|\omega\|_{L^6(B_t\backslash B_s)}^{\frac{3}{p}(1-\theta)}\|u\|_{L^p(B_t\backslash B_s)}\right)^\frac{2p}{3(1-\theta)}.
\endaligned
\end{align}
Substituting the following inequality
$$
\|\omega\|_{L^6(B_t\backslash B_s)}\leq CR^\frac{1}{2}\|\omega\|_{L^\infty(A_R)}
$$
into \eqref{later2.20}, we derive that \eqref{later2.14} holds.

Obviously, \eqref{ine-a3} is a direct consequence of \eqref{later2.16} and \eqref{later2.17}.
\end{proof}

\begin{Lem}\label{Lem2.7}
Let $\sqrt{3}R\leq s<t\leq 2R$. Suppose that $u,b$ are smooth vector-valued functions. Then
\begin{itemize}
\item[(i)] Let $\frac{3}{2}<p\leq\frac{9}{2}$, $1\leq r<2p'$. It holds that
\begin{align}\label{ine2.14}
J_5&\leq\frac{1}{t-s}\|u\|_{L^p(B_t\backslash B_\frac{3R}{2})}\|b\|_{L^r(B_t\backslash B_{\frac{3R}{2}})}^\frac{2(3-p')r}{(6-r)p'}\|b\|_{L^6(B_t\backslash B_\frac{3R}{2})}^\frac{12p'-6r}{(6-r)p'}.
\end{align}
\item[(ii)] Let $\frac{3}{2}<p\leq\frac{9}{2}$, $1\leq r<2p'$. For any $\delta>0$, there exists a positive constant $C_\delta$ such that
\begin{align}\label{ine2.13}
J_5\leq\delta\|b\|_{L^6(B_t\backslash B_\frac{3R}{2})}^2+\frac{C_\delta}{(t-s)^\frac{(6-r)p'}{(3-p')r}}\|u\|_{L^p(A_R)}^\frac{(6-r)p'}{(3-p')r}\|b\|_{L^r(A_R)}^2.
\end{align}
\item[(iii)] Let $\frac{3}{2}<p\leq\frac{9}{2}$, $2p'\leq r<+\infty$. It holds that
\begin{align}\label{ine2.15}
\aligned
J_5&\leq\frac{C}{t-s}R^{3-\frac{3}{p}-\frac{6}{r}}\|b\|_{L^r(A_R)}^{2}\|u\|_{L^p(A_R)}.
\endaligned
\end{align}
\end{itemize}
\end{Lem}

\begin{proof}
When $\frac{3}{2}<p\leq\frac{9}{2}$ and $1\leq r<2p'$, we have $2p'\in(r,6)$. At this time, the following interpolation inequality holds
\begin{equation*}
\|b\|_{L^{2p'}} \leq\|b\|_{L^r}^\frac{(3-p')r}{(6-r)p'}\|b\|_{L^6}^\frac{6p'-3r}{(6-r)p'},
\end{equation*}
which indicates
\begin{equation*}
\aligned
J_5&\leq\frac{1}{t-s}\|u\|_{L^p(B_t\backslash B_\frac{3R}{2})}\|b\|_{L^r(B_t\backslash B_{\frac{3R}{2}})}^\frac{(6-2p')r}{(6-r)p'}\|b\|_{L^6(B_t\backslash B_{\frac{3R}{2}})}^\frac{12p'-6r}{(6-r)p'}.
\endaligned
\end{equation*}
Applying the Young inequality, we get
$$
J_5\leq\delta\|b\|_{L^6(B_t\backslash B_{\frac{3R}{2}})}^2+\frac{C_\delta}{(t-s)^\frac{(6-r)p'}{(3-p')r}}\|u\|_{L^p(B_t\backslash B_\frac{3R}{2})}^\frac{(6-r)p'}{(3-p')r}\|b\|_{L^r(B_t\backslash B_{\frac{3R}{2}})}^2.
$$

When $\frac{3}{2}<p\leq\frac{9}{2}$ and $2p'\leq r<+\infty$,
applying the H\"{o}lder inequality to $J_5$, we obtain
\begin{equation*}
\aligned
J_5&\leq\frac{C}{t-s}R^{\frac{3}{p'}-\frac{6}{r}}\|b\|_{L^r(B_t\backslash B_\frac{3R}{2})}^{2}\|u\|_{L^p(B_t\backslash B_\frac{3R}{2})}\\
&=\frac{C}{t-s}R^{3-\frac{3}{p}-\frac{6}{r}}\|b\|_{L^r(B_t\backslash B_{\frac{3R}{2}})}^{2}\|u\|_{L^p(B_t\backslash B_\frac{3R}{2})}.
\endaligned
\end{equation*}

\end{proof}

\section{Proof of Theorem \ref{main1}}\label{sec3}
In order to prove Theorem \ref{main1}, we first establish an important energy estimate.
\begin{Lem}
Let $(u,\pi,\omega,b)$ be a smooth solution of \eqref{equ1.1} and $\chi\in (0,2)$, $\sqrt{3}R\leq s<t\leq 2R$. Denote
\begin{align}\label{add-f}
f(\rho)=&\int_{B_\rho}\left(|\nabla u|^2+|\nabla\omega|^2+|\nabla b|^2+|\omega|^2\right)dx+\|u\|_{L^6 (B_\rho)}^2+\|\omega\|_{L^6 (B_\rho)}^2+\|b\|_{L^6 (B_\rho)}^2.
\end{align}
Then it holds that
\begin{align}\label{ine3.12}
f(s)\leq&\frac{1}{2}\int_{B_t\backslash B_{\frac{\sqrt{3}t}{2}}}|\nabla u|^2d x+\frac{1}{4}\int_{B_t\backslash B_{\frac{\sqrt{3}t}{2}}}|\omega|^2d x+C(J_1+J_2+J_3+J_4+J_5),
\end{align}
where $J_1$, $J_2$, $J_3$, $J_4$, $J_5$ are given in Section \ref{sec2}.
\end{Lem}
\begin{proof}
Let $\sqrt{3}R\leq s<t\leq 2R$, then we infer $s\geq\frac{\sqrt{3}t}{2}>\frac{3R}{2}$. We introduce a cut-off function $\eta \in C_0^{\infty}\left(\mathbb{R}^{3}\right)$ satisfying
\begin{align*}
\eta(x)= \begin{cases}1, & |x| <s, \\ 0, & |x| >\frac{s+t}{2},\end{cases}
\end{align*}
with
$$\text{$0\leq\eta (x)\leq 1$, and $\|\nabla \eta\|_{L^{\infty}} \leq \frac{C}{t-s}$, $\|\nabla^2\eta\|_{L^{\infty}} \leq \frac{C}{(t-s)^2}$.}$$
Since
$$\int_{B_t\backslash B_\frac{\sqrt{3}t}{2}}u\cdot\nabla\eta^{2}dx=\int_{B_t}u\cdot\nabla\eta^{2}dx=\int_{B_t}\nabla\cdot (u\eta^{2})dx=\int_{\partial B_t}\frac{x}{|x|}\cdot (u\eta^{2})dS=0,$$
by Lemma \ref{Lem2.2}, there exists $v\in W_{0}^{1,\sigma}(B_t\backslash \overline{B_\frac{\sqrt{3}t}{2}})$ such that $v$ satisfies the following equation
\begin{align}\label{v}
\mathrm{div} v=u\cdot\nabla\eta^{2} \text{ in }B_t\backslash \overline{B_\frac{\sqrt{3}t}{2}},
\end{align}
with the estimate
\begin{align}\label{ine2.2}
\|\nabla v\|_{L^\sigma(B_t\backslash B_\frac{\sqrt{3}t}{2})}\leq C\|u\cdot\nabla\eta^2\|_{L^\sigma(B_t\backslash B_\frac{\sqrt{3}t}{2})}\leq\frac{C}{t-s}\|u\|_{L^\sigma(B_t\backslash B_s)},
\end{align}
for any $1<\sigma<+\infty$. We extend $v$ by zero to $B_\frac{\sqrt{3}t}{2}$, then $v\in W_{0}^{1,\sigma}(B_t).$

Thanks to the divergence free condition for the velocity field, it holds that
$$-\Delta u=\text{\rm curl}^2u.$$
Multiplying both sides of $\eqref{equ1.1}_{1}$ by $u \eta^2-v$, where $v$ is a solution of \eqref{v}, then integrating over $B_t$ and using integration by parts, we obtain
\begin{align}\label{ine3.1}
\aligned
&\int_{B_t}|\curl u|^2\eta^2dx=\int_{B_t}\text{\rm curl}u\cdot\curl vdx-\int_{B_t}\text{\rm curl}u\cdot(\nabla\eta^2\times u)dx\\
+&\int_{B_t}\frac{|u|^2}{2}(u\cdot\nabla\eta^2)dx-\int_{B_t}(u\cdot \nabla)v\cdot udx+\chi \int_{B_t}\omega\cdot(\eta^2\curl u+\nabla\eta^2\times u)dx\\
-&\chi \int_{B_t}\omega\cdot \curl  vdx+\int_{B_t}(b\cdot\nabla)v\cdot b-(b\cdot\nabla)u\cdot b\eta^2-(u\cdot b)(b\cdot\nabla\eta^2)dx.
\endaligned
\end{align}
Since the following identity holds
$$\curl u\times u=(u\cdot\nabla)u-\nabla\left(\frac{|u|^2}{2}\right),$$
we have
\begin{align}\label{ine3.2}
\aligned
\int_{B_t}(\text{\rm curl}u\times u)\cdot\nabla\eta^2dx&=\int_{B_t}\left[(u\cdot\nabla)u-\nabla\left(\frac{|u|^2}{2}\right)\right]\cdot\nabla\eta^2dx\\
&=\int_{B_t}\frac{|u|^2}{2}\Delta\eta^2-(u\otimes u):\nabla^2(\eta^2)dx.
\endaligned
\end{align}
Multiplying both sides of $\eqref{equ1.1}_{2}$ by $\omega \eta^2$, then integrating over $B_t$ and using integration by parts, we obtain
\begin{align}\label{ine3.3}
\aligned
&\int_{B_t}|\nabla\omega|^2\eta^2+|\mathrm{div}\omega|^2\eta^2+2\chi|\omega|^2\eta^2dx\\
=&\int_{B_t}\frac{|\omega|^2}{2}\Delta \eta^2+\frac{|\omega|^2}{2}(u\cdot\nabla\eta^2)dx-\int_{B_t}\mathrm{div}\omega(\omega\cdot\nabla\eta^2)dx+\chi\int_{B_t}\curl  u\cdot \omega\eta^2dx.
\endaligned
\end{align}
Multiplying both sides of $\eqref{equ1.1}_{3}$ by $b \eta^2$, then integrating over $B_t$ and using integration by parts, we obtain
\begin{align}\label{ine3.4}
\int_{B_t}|\nabla b|^2\eta^2dx=\int_{B_t}\frac{|b|^2}{2}\Delta \eta^2+\frac{|b|^2}{2}(u\cdot\nabla\eta^2)dx+\int_{B_t}(b\cdot\nabla)u\cdot b\eta^2dx.
\end{align}
Combining \eqref{ine3.1}, \eqref{ine3.2}, \eqref{ine3.3} and \eqref{ine3.4}, we find that
\begin{align}\label{ine3.5}
&\int_{B_t}\left(|\curl u|^2+|\nabla \omega|^{2}+|\mathrm{div}\omega|^2+|\nabla b|^{2}+2\chi|\omega|^2\right)\eta^2dx\notag\\
=&\int_{B_t}\text{\rm curl}u\cdot\curl vdx+2\chi\int_{B_t}\curl  u\cdot \omega\eta^2dx-\int_{B_t}\mathrm{div}\omega(\omega\cdot\nabla\eta^2)dx\notag\\
&+\chi \int_{B_t}\omega\cdot(\nabla\eta^2\times u)dx-\chi \int_{B_t}\omega\cdot \curl  vdx+\frac{1}{2}\int_{B_t}(|u|^2+|\omega|^2+|b|^2) \Delta \eta^2dx\\
&-\int_{B_t}(u\otimes u):\nabla^2(\eta^2)dx+\frac{1}{2} \int_{B_t}(|u|^2+|\omega|^2+|b|^2) (u \cdot \nabla \eta^2) dx\notag\\
&-\int_{B_t}(u\cdot b)(b\cdot\nabla\eta^2)dx +\int_{B_t}(b\otimes b-u\otimes u): \nabla v  dx.\notag
\end{align}
Let $\varepsilon$ and $\delta$ be two positive constants to be determined. Using the Young inequality, we get
\begin{align}\label{ine3.6}
&2\chi\int_{B_t}\curl  u\cdot \omega\eta^2dx-\int_{B_t}\mathrm{div}\omega(\omega\cdot\nabla\eta^2)dx+\chi \int_{B_t}\omega\cdot(\nabla\eta^2\times u)dx\notag\\
\leq&\varepsilon\int_{B_t}|\curl  u|^2\eta^2dx+\frac{\chi^2}{\varepsilon}\int_{B_t}|\omega|^2\eta^2dx+\frac{1}{2}\int_{B_t}|\mathrm{div}\omega|^2\eta^2dx+2\int_{B_t}|\omega\cdot\nabla\eta|^2dx\\
&+\delta\int_{B_t}|\omega|^2\eta^2dx+\frac{\chi^2}{\delta}\int_{B_t}|\nabla\eta\times u|^2dx.\notag
\end{align}
Combining \eqref{ine3.5} and \eqref{ine3.6}, we can derive that
\begin{align*}
&\int_{B_t}\left[(1-\varepsilon)|\curl u|^2+|\nabla \omega|^{2}+\frac{1}{2}|\mathrm{div}\omega|^2+|\nabla b|^{2}+\left(2\chi-\frac{\chi^2}{\varepsilon}-\delta\right)|\omega|^2\right]\eta^2dx\\
\leq&\int_{B_t}\text{\rm curl}u\cdot\curl vdx-\chi \int_{B_t}\omega\cdot \curl  vdx+\frac{C}{(t-s)^2}\int_{B_t \backslash B_s}(|u|^2+|\omega|^2+|b|^2)dx\\
&+\frac{C}{t-s} \int_{B_t\backslash B_s}(|u|^2+|\omega|^2+|b|^2) |u|dx+\int_{B_t}(|u|^2+|b|^2)|\nabla v| dx.
\end{align*}
Since $\chi\in(0,2)$, we can choose some $\varepsilon$ and $\delta$ satisfying
$$\varepsilon\in\left(\frac{\chi}{2},1\right),\delta\in\left(0,2\chi-\frac{\chi^2}{\varepsilon}\right).$$
Hence, it follows that
\begin{align}\label{ine3.7}
&\int_{B_t}\left(|\curl u|^2+|\nabla \omega|^{2}+|\mathrm{div}\omega|^2+|\nabla b|^{2}+|\omega|^2\right)\eta^2dx\notag\\
\leq&C\int_{B_t\backslash B_{\frac{\sqrt{3}t}{2}}}\left(|\nabla u|\cdot|\nabla v|+|\omega|\cdot|\nabla v|\right)dx+\frac{C}{(t-s)^2}\int_{B_t \backslash B_s}(|u|^2+|\omega|^2+|b|^2)dx\\
&+\frac{C}{t-s} \int_{B_t\backslash B_s}(|u|^2+|\omega|^2+|b|^2) |u|dx+C\int_{B_t}(|u|^2+|b|^2)|\nabla v| dx.\notag
\end{align}
In view of the identity $\text{\rm curl}^2\psi=-\Delta\psi+\nabla(\divg \psi)$ and the divergence free condition, it suffices to get
\begin{align}\label{ine3.8}
\int_{B_t}|\nabla(u\eta)|^2 d x&=\int_{B_t}|\divg (u\eta)|^2 d x+\int_{B_t}|\text{\rm curl}(u\eta)|^2 d x\notag\\
&=\int_{B_t}|u\cdot\nabla\eta|^2 d x+\int_{B_t}|\eta\text{\rm curl}u+\nabla\eta\times u|^2 d x\\
&\leq\int_{B_t}|u\cdot\nabla\eta|^2 d x+2\int_{B_t}|\text{\rm curl}u|^2\eta^2dx+2\int_{B_t}|\nabla\eta\times u|^2 d x.\notag
\end{align}
Using the Gagliardo-Nirenberg inequality and \eqref{ine3.8}, we have
\begin{align}\label{ine3.9}
&\|u \eta\|_{L^6 (B_t)}^2+\|\omega \eta\|_{L^6 (B_t)}^2+\|b \eta\|_{L^6 (B_t)}^2\notag\\
\leq&C\left(\|\nabla(u \eta)\|_{L^2 (B_t)}^2+\|\nabla(\omega \eta)\|_{L^2 (B_t)}^2+\|\nabla(b \eta)\|_{L^2 (B_t)}^2 \right)\notag\\
\leq&C \left(\|\eta \text{\rm curl}u\|_{L^2 (B_t)}^2+\|\eta\nabla \omega\|_{L^2 (B_t)}^2+\|\eta\nabla b\|_{L^2 (B_t)}^2\right)\\
&+C\left(\|u\otimes\nabla \eta\|_{L^2 (B_t)}^2+\|\omega\otimes\nabla \eta\|_{L^2 (B_t)}^2+\|b\otimes\nabla \eta\|_{L^2 (B_t)}^2 \right).\notag
\end{align}

Combining \eqref{ine3.7} and \eqref{ine3.9}, and using the Young inequality, we have
\begin{align*}
&\int_{B_t}\left(|\nabla(u\eta)|^2+|\nabla(\omega\eta)|^2+|\nabla(b\eta)|^2+|\omega \eta|^2\right)dx+\|u \eta\|_{L^6 (B_t)}^2+\|\omega \eta\|_{L^6 (B_t)}^2+\|b\eta\|_{L^6 (B_t)}^2\\
\leq&\frac{1}{2}\int_{B_t\backslash B_{\frac{\sqrt{3}t}{2}}}|\nabla u|^2d x+\frac{1}{4}\int_{B_t\backslash B_{\frac{\sqrt{3}t}{2}}}|\omega|^2d x+C\int_{B_t\backslash B_{\frac{\sqrt{3}t}{2}}}|\nabla v|^2d x\\
&+\frac{C}{(t-s)^2}\int_{B_t \backslash B_s}(|u|^2+|\omega|^2+|b|^2) d x+\frac{C}{t-s} \int_{B_t\backslash B_s}(|u|^2+|\omega|^2+|b|^2) |u|dx\\
&+C\int_{B_t}(|u|^2+|b|^2)|\nabla v| dx.
\end{align*}
Hence, it follows that
\begin{align}\label{ine3.10}
f(s)\leq&\frac{1}{2}\int_{B_t\backslash B_{\frac{\sqrt{3}t}{2}}}|\nabla u|^2d x+\frac{1}{4}\int_{B_t\backslash B_{\frac{\sqrt{3}t}{2}}}|\omega|^2d x+\frac{C}{(t-s)^2}\int_{B_t \backslash B_s}(|u|^2+|\omega|^2+|b|^2) d x\notag\\
&+\frac{C}{t-s} \int_{B_t\backslash B_s}(|u|^2+|\omega|^2+|b|^2) |u|dx+C\int_{B_t\backslash B_{\frac{\sqrt{3}t}{2}}}(|u|^2+|b|^2)|\nabla v|dx.
\end{align}
Here we have used \eqref{ine2.2}. By the H\"{o}lder inequality and \eqref{ine2.2}, we have
\begin{align}\label{ine3.11}
&\frac{C}{t-s} \int_{B_t\backslash B_s}|u|^3dx+C\int_{B_t\backslash B_{\frac{\sqrt{3}t}{2}}}|u|^2|\nabla v| dx\notag\\
\leq&\frac{C}{t-s}\|u\|_{L^{3}(B_t\backslash B_s)}^3+\|u\|_{L^{3}(B_t\backslash B_{\frac{\sqrt{3}t}{2}})}^2\|\nabla v\|_{L^{3}(B_t\backslash B_{\frac{\sqrt{3}t}{2}})}\\
\leq&\frac{C}{t-s}\|u\|_{L^{3}(B_t\backslash B_{\frac{3R}{2}})}^3,\notag
\end{align}
and
\begin{align}\label{add-ub2}
&\frac{C}{t-s} \int_{B_t\backslash B_s}|b|^2|u|dx+C\int_{B_t\backslash B_{\frac{\sqrt{3}t}{2}}}|b|^2|\nabla v| dx\notag\\
\leq&\frac{C}{t-s}\|b\|_{L^{2p'}(B_t\backslash B_s)}^2\|u\|_{L^p(B_t\backslash B_s)}+\|b\|_{L^{2p'}(B_t\backslash B_{\frac{\sqrt{3}t}{2}})}^2\|\nabla v\|_{L^p(B_t\backslash B_{\frac{\sqrt{3}t}{2}})}\\
\leq&\frac{C}{t-s}\|b\|_{L^{2p'}(B_t\backslash B_{\frac{3R}{2}})}^2\|u\|_{L^p(B_t\backslash B_{\frac{3R}{2}})}.\notag
\end{align}
Combining \eqref{ine3.10}, \eqref{ine3.11} and \eqref{add-ub2}, and recalling the definitions of the terms $J_1$, $J_2$, $J_3$, $J_4$, $J_5$ in Section \ref{sec2}, we can verify that
\eqref{ine3.12} holds.
\end{proof}

Now we are in a position to prove Theorem \ref{main1}.

\begin{proof}[{\bf Proof of Theorem \ref{main1}}]
We divide the assumptions into two main cases, i.e., $\frac{3}{2} <p<3$ and $3 \leq p\leq\frac{9}{2}$.
Firstly, we consider the case $\frac{3}{2} <p<3$. Since
 $$\liminf\limits_{R\rightarrow+\infty}\left[X_{p,\alpha}(R)+Y_{q,\beta}(R)+Z_{r,\gamma}(R)\right]<+\infty,$$
 there exists a sequence $R_j\nearrow+\infty$ such that
\begin{align}\label{lim}
\lim\limits_{j\rightarrow+\infty}X_{p,\alpha}(R_j)<+\infty,\;\lim\limits_{j\rightarrow+\infty}Y_{q,\beta}(R_j)<+\infty,\;\lim\limits_{j\rightarrow+\infty}Z_{r,\gamma}(R_j)<+\infty.
\end{align}

\textbf{Assume that (A1) holds.} We first consider the case $1\leq q<+\infty$. Combining \eqref{ine3.12}, \eqref{ine2.3}, \eqref{later2.5}, \eqref{ine2.7}, \eqref{later2.13} and \eqref{ine2.13}, we derive that
\begin{align*}
f(s)\leq&\frac{1}{2}f(t)+\frac{CR^{\frac{6}{p}-1}}{(t-s)^{\frac{12}{p}-2}}\|u\|_{L^p(A_R)}^{2}+\frac{C}{(t-s)^2}R^{3-\frac{6}{p}}\|u\|_{L^{p}(A_R)}^{2}+\frac{CR^{\frac{6}{r}-1}}{(t-s)^{\frac{12}{r}-2}}\|b\|_{L^{r}(A_R)}^{2}\\
&+\frac{C}{(t-s)^2}R^{3-\frac{6}{r}}\|b\|_{L^r(A_R)}^{2}+\left(\frac{Ch_1(q)}{(t-s)^\frac{6-q}{q(1-\theta)}}+\frac{CR^{3-\frac{6}{q}}}{(t-s)^\frac{2}{1-\theta}}\right)\|\omega\|_{L^q(A_R)}^2\\
&+\frac{C}{(t-s)^\frac{6-p}{2p-3}}\|u\|_{L^p(A_R)}^\frac{3p}{2p-3}+Ch_2(q)\left(\frac{1}{t-s}\|\omega\|_{L^q(A_R)}^{\frac{(6-2p')q}{(6-q)p'}(1-\theta)}\|u\|_{L^p(A_R)}\right)^{\frac{(6-q)p'}{(3-p')q(1-\theta)}}\\
&+Ch_3(q)\left(\frac{1}{t-s}\|\omega\|_{L^q(A_R)}^{\frac{2q}{(q-2)p}(1-\theta)}\|u\|_{L^p(A_R)}\right)^\frac{(q-2)p}{(1-\theta)q}+\frac{C}{(t-s)^\frac{(6-r)p'}{(3-p')r}}\|u\|_{L^p(A_R)}^\frac{(6-r)p'}{(3-p')r}\|b\|_{L^r(A_R)}^2,
\end{align*}
where $\theta\in(0,1)$ is to be determined.

Applying Lemma \ref{Lem2.3} to the above function inequality, and taking $s=\sqrt{3}R$ and $t=2R$, we conclude that
\begin{align*}
f(R)\leq& f\left(\sqrt{3}R\right)\leq CR^{1-\frac{6}{p}}\|u\|_{L^p(A_R)}^2+CR^{1-\frac{6}{r}}\|b\|_{L^{r}(A_R)}^{2}\\
&+C\left(h_1(q)R^{-\frac{6-q}{q(1-\theta)}}+R^{3-\frac{6}{q}-\frac{2}{1-\theta}}\right)\|\omega\|_{L^q(A_R)}^2+CR^\frac{p-6}{2p-3}\|u\|_{L^p(A_R)}^\frac{3p}{2p-3}\\
&+Ch_2(q)\left(\frac{1}{R}\|\omega\|_{L^q(A_R)}^{\frac{(6-2p')q}{(6-q)p'}(1-\theta)}\|u\|_{L^p(A_R)}\right)^{\frac{(6-q)p'}{(3-p')q(1-\theta)}}\\
&+Ch_3(q)\left(\frac{1}{R}\|\omega\|_{L^q(A_R)}^{\frac{2q}{(q-2)p}(1-\theta)}\|u\|_{L^p(A_R)}\right)^\frac{(q-2)p}{(1-\theta)q}+CR^{-\frac{(6-r)p'}{(3-p')r}}\|u\|_{L^p(A_R)}^\frac{(6-r)p'}{(3-p')r}\|b\|_{L^r(A_R)}^2.
\end{align*}
Hence, it holds that
\begin{align}\label{add1}
f(R)\leq&CR^{1-\frac{6}{p}+2\alpha}[X_{p,\alpha}(R)]^2+CR^{1-\frac{6}{r}+2\gamma}[Z_{r,\gamma}(R)]^2+CR^{\frac{p-6}{2p-3}+\frac{3p}{2p-3}\alpha}[X_{p,\alpha}(R)]^\frac{3p}{2p-3}\notag\\
&+C\left(h_1(q)R^{-\frac{6-q}{q(1-\theta)}+2\beta}+R^{3-\frac{6}{q}-\frac{2}{1-\theta}+2\beta}\right)[Y_{q,\beta}(R)]^2\notag\\
&+Ch_2(q)\left(R^{\alpha+\frac{(6-2p')q}{(6-q)p'}(1-\theta)\beta-1}X_{p,\alpha}(R)[Y_{q,\beta}(R)]^{\frac{(6-2p')q}{(6-q)p'}(1-\theta)}\right)^{\frac{(6-q)p'}{(3-p')q(1-\theta)}}\\
&+Ch_3(q)\left(R^{\alpha+\frac{2q}{(q-2)p}(1-\theta)\beta-1}X_{p,\alpha}(R)[Y_{q,\beta}(R)]^{\frac{2q}{(q-2)p}(1-\theta)}\right)^\frac{(q-2)p}{(1-\theta)q}\notag\\
&+CR^{-\frac{(6-r)p'}{(3-p')r}+\frac{(6-r)p'}{(3-p')r}\alpha+2\gamma}X_{p,\alpha}(R)^\frac{(6-r)p'}{(3-p')r}[Z_{r,\gamma}(R)]^2.\notag
\end{align}
It is easy to verify that for any fixed $\beta$, we can choose a positive constant $\theta$ sufficiently close to $1$ such that
$$-\frac{6-q}{q(1-\theta)}+2\beta<0\;(\text{when }q<6),\;3-\frac{6}{q}-\frac{2}{1-\theta}+2\beta<0,$$
$$\alpha+\frac{(6-2p')q}{(6-q)p'}(1-\theta)\beta-1<0\;(\text{when }q<6),\;\alpha+\frac{2q}{(q-2)p}(1-\theta)\beta-1<0\;(\text{when }q>2).$$
Letting $R=R_j\rightarrow+\infty$, thanks to \eqref{lim} and
$$1-\frac{6}{p}+2\alpha<0,\;1-\frac{6}{r}+2\gamma\leq0,\;\frac{p-6}{2p-3}+\frac{3p}{2p-3}\alpha\leq0,$$
$$-\frac{(6-r)p'}{(3-p')r}+\frac{(6-r)p'}{(3-p')r}\alpha+2\gamma=\frac{(6-r)p}{(2p-3)r}\left(-1+\alpha+\frac{(4p-6)r}{(6-r)p}\gamma\right)\leq0,$$
we get $u,\omega,b\in L^6(\mathbb{R}^3)$ and $\omega,\nabla u,\nabla \omega,\nabla b\in L^2(\mathbb{R}^3)$. Furthermore, it holds that
\begin{align}\label{ine3.13}
\aligned
&\lim_{R\rightarrow+\infty}\left(\|u\|_{L^6(A_R)}+\|\omega\|_{L^6(A_R)}+\|b\|_{L^6(A_R)}+\|\omega\|_{L^2(A_R)}\right)=0,\\
&\lim_{R\rightarrow+\infty} \left(\|\nabla u\|_{L^2(A_R)}+\|\nabla \omega\|_{L^2(A_R)}+\|\nabla b\|_{L^2(A_R)}\right)=0.
\endaligned
\end{align}
Combining \eqref{ine3.12}, \eqref{ine2.4}, \eqref{later2.7}, \eqref{ine2.6}, \eqref{ine-a3} and \eqref{ine2.14}, we have
\begin{align*}
f(s)\leq&\frac{1}{2}\|\nabla u\|_{L^2(A_R)}^2+\frac{1}{4}\|\omega\|_{L^2(A_R)}^2+\frac{CR^2}{(t-s)^2}\left(\|u\|_{L^6(A_R)}^2+\|b\|_{L^6(A_R)}^2\right)\\
&+\frac{C}{(t-s)^2}\|\omega\|_{L^2(A_R)}^2+\frac{C}{t-s}\|u\|_{L^p(A_R)}^{\frac{3p}{6-p}}\|u\|_{L^6(A_R)}^{\frac{18-6p}{6-p}}\\
&+\frac{C}{t-s}\|u\|_{L^p(A_R)}\|\omega\|_{L^2(A_R)}^{2-\frac{3}{p}}\|\omega\|_{L^6(A_R)}^\frac{3}{p}+\frac{C}{t-s}\|u\|_{L^p(A_R)}\|b\|_{L^r(A_R)}^\frac{2(3-p')r}{(6-r)p'}\|b\|_{L^6(A_R)}^\frac{12p'-6r}{(6-r)p'}.
\end{align*}
Taking $s=\sqrt{3}R$ and $t=2R$ in the above inequality, we get
\begin{align*}
f(R)\leq &f\left(\sqrt{3}R\right)\leq \frac{1}{2}\|\nabla u\|_{L^2(A_R)}^2+\frac{1}{4}\|\omega\|_{L^2(A_R)}^2+C\left(\|u\|_{L^6(A_R)}^2+\|b\|_{L^6(A_R)}^2\right)\\
&+CR^{-2}\|\omega\|_{L^2(A_R)}^2+CR^{\frac{3p}{6-p}\alpha-1}[X_{p,\alpha}(R)]^{\frac{3p}{6-p}}\|u\|_{L^6(A_R)}^{\frac{18-6p}{6-p}}\\
&+CR^{\alpha-1}X_{p,\alpha}(R)\|\omega\|_{L^2(A_R)}^{2-\frac{3}{p}}\|\omega\|_{L^6(A_R)}^\frac{3}{p}+CR^{-1+\alpha+\frac{2(3-p')r}{(6-r)p'}\gamma}X_{p,\alpha}(R)\|b\|_{L^6(A_R)}^\frac{12p'-6r}{(6-r)p'}.
\end{align*}
Letting $R=R_j\rightarrow+\infty$ and thanks to \eqref{ine3.13}, we obtain that $u=\omega=b=0$.

\par Now we consider the case $q=+\infty$. Routinely,  combining \eqref{ine3.12}, \eqref{ine2.3}, \eqref{later2.6}, \eqref{ine2.7}, \eqref{later2.14} and \eqref{ine2.13}, and applying Lemma \ref{Lem2.3}, we deduce that
\begin{align}\label{add-f1}
f(R)\leq&CR^{1-\frac{6}{p}+2\alpha}[X_{p,\alpha}(R)]^2+CR^{1-\frac{6}{r}+2\gamma}[Z_{r,\gamma}(R)]^2\notag\\
&+CR^{3-\frac{2}{1-\theta}+2\beta}[Y_{q,\beta}(R)]^2+CR^{\frac{p-6}{2p-3}+\frac{3p}{2p-3}\alpha}[X_{p,\alpha}(R)]^\frac{3p}{2p-3}\notag\\
&+C\left(R^{\frac{3}{2p}(1-\theta)+\alpha+\frac{3}{p}(1-\theta)\beta-1}X_{p,\alpha}(R)[Y_{q,\beta}(R)]^{\frac{3}{p}(1-\theta)}\right)^\frac{2p}{3(1-\theta)}\\
&+CR^{-\frac{(6-r)p'}{(3-p')r}+\frac{(6-r)p'}{(3-p')r}\alpha+2\gamma}[X_{p,\alpha}(R)]^\frac{(6-r)p'}{(3-p')r}[Z_{r,\gamma}(R)]^2.\notag
\end{align}
Hence, we get $u,\omega,b\in L^6(\mathbb{R}^3)$ and $\omega,\nabla u,\nabla \omega,\nabla b\in L^2(\mathbb{R}^3)$. The rest of the proof is the same with the case $1\leq q<+\infty$.

\textbf{Assume that (A2) holds.}  We first consider the case $1\leq q<+\infty$.
Routinely, combining \eqref{ine3.12}, \eqref{ine2.3}, \eqref{later2.5}, \eqref{ine2.7}, \eqref{later2.13} and \eqref{ine2.15}, and applying Lemma \ref{Lem2.3}, we derive that
\begin{align}\label{add2}
f(R)\leq&CR^{1-\frac{6}{p}+2\alpha}[X_{p,\alpha}(R)]^2+CR^{1-\frac{6}{r}+2\gamma}[Z_{r,\gamma}(R)]^2+CR^{\frac{p-6}{2p-3}+\frac{3p}{2p-3}\alpha}[X_{p,\alpha}(R)]^\frac{3p}{2p-3}\notag\\
&+C\left(h_1(q)R^{-\frac{6-q}{q(1-\theta)}+2\beta}+R^{3-\frac{6}{q}-\frac{2}{1-\theta}+2\beta}\right)[Y_{q,\beta}(R)]^2\notag\\
&+Ch_2(q)\left(R^{\alpha+\frac{(6-2p')q}{(6-q)p'}(1-\theta)\beta-1}X_{p,\alpha}(R)[Y_{q,\beta}(R)]^{\frac{(6-2p')q}{(6-q)p'}(1-\theta)}\right)^{\frac{(6-q)p'}{(3-p')q(1-\theta)}}\\
&+Ch_3(q)\left(R^{\alpha+\frac{2q}{(q-2)p}(1-\theta)\beta-1}X_{p,\alpha}(R)[Y_{q,\beta}(R)]^{\frac{2q}{(q-2)p}(1-\theta)}\right)^\frac{(q-2)p}{(1-\theta)q}\notag\\
&+CR^{2-\frac{3}{p}-\frac{6}{r}
+\alpha+2\gamma}X_{p,\alpha}(R)[Z_{r,\gamma}(R)]^2.\notag
\end{align}
Hence, we get $u,\omega,b\in L^6(\mathbb{R}^3)$ and $\omega,\nabla u,\nabla \omega,\nabla b\in L^2(\mathbb{R}^3)$.
Routinely, combining \eqref{ine3.12}, \eqref{ine2.4}, \eqref{later2.7}, \eqref{ine2.6}, \eqref{ine-a3} and \eqref{ine2.15}, we have
\begin{align*}
f(R)&\leq \frac{1}{2}\|\nabla u\|_{L^2(A_R)}^2+\frac{1}{4}\|\omega\|_{L^2(A_R)}^2+CR^{\frac{3p}{6-p}\alpha-1}[X_{p,\alpha}(R)]^\frac{3p}{6-p}\|u\|_{L^6(A_R)}^{\frac{18-6p}{6-p}}\\
&+C\left(\|u\|_{L^6(A_R)}^2+\|b\|_{L^6(A_R)}^2+R^{-2}\|\omega\|_{L^2(A_R)}^2\right)\\
&+CR^{\alpha-1}X_{p,\alpha}(R)\|\omega\|_{L^2(A_R)}^{2-\frac{3}{p}}\|\omega\|_{L^6(A_R)}^\frac{3}{p}+CR^{2-\frac{3}{p}-\frac{6}{r}+\alpha+2\gamma}X_{p,\alpha}(R)[Z_{r,\gamma}(R)]^2.
\end{align*}
Consequently, $u=\omega=b=0$. Now we consider the case $q=+\infty$. Routinely,  combining \eqref{ine3.12}, \eqref{ine2.3}, \eqref{later2.6}, \eqref{ine2.7}, \eqref{later2.14} and \eqref{ine2.15}, and applying Lemma \ref{Lem2.3}, we deduce that
\begin{align}\label{add-f2}
f(R)\leq&CR^{1-\frac{6}{p}+2\alpha}[X_{p,\alpha}(R)]^2+CR^{1-\frac{6}{r}+2\gamma}[Z_{r,\gamma}(R)]^2\notag\\
&+CR^{3-\frac{2}{1-\theta}+2\beta}[Y_{q,\beta}(R)]^2+CR^{\frac{p-6}{2p-3}+\frac{3p}{2p-3}\alpha}[X_{p,\alpha}(R)]^\frac{3p}{2p-3}\notag\\
&+C\left(R^{\frac{3}{2p}(1-\theta)+\alpha+\frac{3}{p}(1-\theta)\beta-1}X_{p,\alpha}(R)[Y_{q,\beta}(R)]^{\frac{3}{p}(1-\theta)}\right)^\frac{2p}{3(1-\theta)}\\
&+CR^{2-\frac{3}{p}-\frac{6}{r}+\alpha+2\gamma}X_{p,\alpha}(R)[Z_{r,\gamma}(R)]^2.\notag
\end{align}
Hence, we get $u,\omega,b\in L^6(\mathbb{R}^3)$ and $\omega,\nabla u,\nabla \omega,\nabla b\in L^2(\mathbb{R}^3)$. The rest of the proof is the same with the case $1\leq q<+\infty$.
\par Next, we consider the case $3 \leq p\leq\frac{9}{2}$.
Since
 $$
 \liminf\limits_{R\rightarrow+\infty}X_{p,\alpha}(R)=0,\;\limsup\limits_{R\rightarrow+\infty}[Y_{q,\beta}(R)+Z_{r,\gamma}(R)]<+\infty,
 $$
there exists a sequence $R_j\nearrow+\infty$ such that
\begin{align}\label{ine3.14}
 \lim\limits_{j\rightarrow+\infty}X_{p,\alpha}(R_j)=0,\;\lim\limits_{j\rightarrow+\infty}Y_{q,\beta}(R_j)<+\infty,\;\lim\limits_{j\rightarrow+\infty}Z_{r,\gamma}(R_j)<+\infty
\end{align}
\par \textbf{Assume that (A3) holds.} We first consider the case $1\leq q<+\infty$. Combining \eqref{ine3.12}, \eqref{ine2.3}, \eqref{later2.5}, \eqref{ine2.5}, \eqref{later2.13} and \eqref{ine2.13}, and applying Lemma \ref{Lem2.3}, we deduce that
\begin{align}\label{add3}
f(R)\leq&CR^{1-\frac{6}{p}+2\alpha}[X_{p,\alpha}(R)]^2+CR^{1-\frac{6}{r}+2\gamma}[Z_{r,\gamma}(R)]^2+CR^{2-\frac{9}{p}+3\alpha}[X_{p,\alpha}(R)]^3\notag\\
&+C\left(h_1(q)R^{-\frac{6-q}{q(1-\theta)}+2\beta}+R^{3-\frac{6}{q}-\frac{2}{1-\theta}+2\beta}\right)[Y_{q,\beta}(R)]^2\notag\\
&+Ch_2(q)\left(R^{\alpha+\frac{(6-2p')q}{(6-q)p'}(1-\theta)\beta-1}X_{p,\alpha}(R)[Y_{q,\beta}(R)]^{\frac{(6-2p')q}{(6-q)p'}(1-\theta)}\right)^{\frac{(6-q)p'}{(3-p')q(1-\theta)}}\\
&+Ch_3(q)\left(R^{\alpha+\frac{2q}{(q-2)p}(1-\theta)\beta-1}X_{p,\alpha}(R)[Y_{q,\beta}(R)]^{\frac{2q}{(q-2)p}(1-\theta)}\right)^\frac{(q-2)p}{(1-\theta)q}\notag\\
&+CR^{-\frac{(6-r)p'}{(3-p')r}+\frac{(6-r)p'}{(3-p')r}\alpha+2\gamma}X_{p,\alpha}(R)^\frac{(6-r)p'}{(3-p')r}[Z_{r,\gamma}(R)]^2\notag.
\end{align}
Letting $R=R_j\rightarrow+\infty$ and using \eqref{ine3.14},
we get $u,\omega,b\in L^6(\mathbb{R}^3)$ and $\omega,\nabla u,\nabla \omega,\nabla b\in L^2(\mathbb{R}^3)$.
Combining \eqref{ine3.12}, \eqref{ine2.4}, \eqref{later2.7}, \eqref{ine2.5}, \eqref{ine-a3} and \eqref{ine2.13}, we have
\begin{align*}
f(R)\leq &\frac{1}{2}\|\nabla u\|_{L^2(A_R)}^2+\frac{1}{4}\|\omega\|_{L^2(A_R)}^2+C\left(\|u\|_{L^6(A_R)}^2+\|b\|_{L^6(A_R)}^2+R^{-2}\|\omega\|_{L^2(A_R)}^2\right)\\
&+CR^{2-\frac{9}{p}+3\alpha}[X_{p,\alpha}(R)]^3+CR^{\alpha-1}X_{p,\alpha}(R)\|\omega\|_{L^2(A_R)}^{2-\frac{3}{p}}\|\omega\|_{L^6(A_R)}^\frac{3}{p}\\
&+CR^{-\frac{(6-r)p'}{(3-p')r}+\frac{(6-r)p'}{(3-p')r}\alpha+2\gamma}[X_{p,\alpha}(R)]^\frac{(6-r)p'}{(3-p')r}[Z_{r,\gamma}(R)]^2.
\end{align*}
Letting $R=R_j\rightarrow+\infty$ and considering \eqref{ine3.13}, we obtain that $u=\omega=b=0$.

Now we consider the case $q=+\infty$. Routinely,  combining \eqref{ine3.12}, \eqref{ine2.3}, \eqref{later2.6}, \eqref{ine2.5}, \eqref{later2.14} and \eqref{ine2.13}, and applying Lemma \ref{Lem2.3}, we deduce that
\begin{align}\label{add-f3}
f(R)\leq& CR^{1-\frac{6}{p}+2\alpha}[X_{p,\alpha}(R)]^2+CR^{1-\frac{6}{r}+2\gamma}[Z_{r,\gamma}(R)]^2\notag\\
&+CR^{3-\frac{2}{1-\theta}+2\beta}[Y_{q,\beta}(R)]^2+CR^{2-\frac{9}{p}+3\alpha}[X_{p,\alpha}(R)]^3\notag\\
&+C\left(R^{\frac{3}{2p}(1-\theta)+\alpha+\frac{3}{p}(1-\theta)\beta-1}X_{p,\alpha}(R)[Y_{q,\beta}(R)]^{\frac{3}{p}(1-\theta)}\right)^\frac{2p}{3(1-\theta)}\\
&+CR^{-\frac{(6-r)p'}{(3-p')r}+\frac{(6-r)p'}{(3-p')r}\alpha+2\gamma}[X_{p,\alpha}(R)]^\frac{(6-r)p'}{(3-p')r}[Z_{r,\gamma}(R)]^2.\notag
\end{align}
Hence, we get $u,\omega,b\in L^6(\mathbb{R}^3)$ and $\omega,\nabla u,\nabla \omega,\nabla b\in L^2(\mathbb{R}^3)$. The rest of the proof is the same with the case $1\leq q<+\infty$.

\textbf{Assume that (A4) holds.} We first consider the case $1\leq q<+\infty$. Routinely, combining \eqref{ine3.12}, \eqref{ine2.3}, \eqref{later2.5}, \eqref{ine2.5}, \eqref{later2.13} and \eqref{ine2.15}, and applying Lemma \ref{Lem2.3}, we deduce that
\begin{align*}
f(R)&\leq CR^{1-\frac{6}{p}+2\alpha}[X_{p,\alpha}(R)]^2+CR^{1-\frac{6}{r}+2\gamma}[Z_{r,\gamma}(R)]^2+CR^{2-\frac{9}{p}+3\alpha}[X_{p,\alpha}(R)]^3\\
&+C\left(h_1(q)R^{-\frac{6-q}{q(1-\theta)}+2\beta}+R^{3-\frac{6}{q}-\frac{2}{1-\theta}+2\beta}\right)[Y_{q,\beta}(R)]^2\\
&+Ch_2(q)\left(R^{\alpha+\frac{(6-2p')q}{(6-q)p'}(1-\theta)\beta-1}X_{p,\alpha}(R)[Y_{q,\beta}(R)]^{\frac{(6-2p')q}{(6-q)p'}(1-\theta)}\right)^{\frac{(6-q)p'}{(3-p')q(1-\theta)}}\\
&+Ch_3(q)\left(R^{\alpha+\frac{2q}{(q-2)p}(1-\theta)\beta-1}X_{p,\alpha}(R)[Y_{q,\beta}(R)]^{\frac{2q}{(q-2)p}(1-\theta)}\right)^\frac{(q-2)p}{(1-\theta)q}\\
&+CR^{2-\frac{3}{p}-\frac{6}{r}+\alpha+2\gamma}X_{p,\alpha}(R)[Z_{r,\gamma}(R)]^2.
\end{align*}
Hence, we get $u,\omega,b\in L^6(\mathbb{R}^3)$ and $\omega,\nabla u,\nabla \omega,\nabla b\in L^2(\mathbb{R}^3)$.
Combining \eqref{ine3.12}, \eqref{ine2.4}, \eqref{later2.7}, \eqref{ine2.5}, \eqref{ine-a3} and \eqref{ine2.15}, we have
\begin{align*}
f(R)\leq &\frac{1}{2}\|\nabla u\|_{L^2(A_R)}^2+\frac{1}{4}\|\omega\|_{L^2(A_R)}^2+C\left(\|u\|_{L^6(A_R)}^2+\|b\|_{L^6(A_R)}^2+R^{-2}\|\omega\|_{L^2(A_R)}^2\right)\\
&+CR^{2-\frac{9}{p}+3\alpha}[X_{p,\alpha}(R)]^3+CR^{\alpha-1}X_{p,\alpha}(R)\|\omega\|_{L^2(A_R)}^{2-\frac{3}{p}}\|\omega\|_{L^6(A_R)}^\frac{3}{p}\\
&+CR^{2-\frac{3}{p}-\frac{6}{r}+\alpha+2\gamma}X_{p,\alpha}(R)[Z_{r,\gamma}(R)]^2.
\end{align*}
Consequently, $u=\omega=b=0$.
 \par Now we consider the case $q=+\infty$. Routinely,  combining \eqref{ine3.12}, \eqref{ine2.3}, \eqref{later2.6}, \eqref{ine2.5}, \eqref{later2.14} and \eqref{ine2.15}, and applying Lemma \ref{Lem2.3}, we deduce that
\begin{align*}
f(R)\leq& CR^{1-\frac{6}{p}+2\alpha}[X_{p,\alpha}(R)]^2+CR^{1-\frac{6}{r}+2\gamma}[Z_{r,\gamma}(R)]^2+CR^{3-\frac{2}{1-\theta}+2\beta}[Y_{q,\beta}(R)]^2\\
&+CR^{2-\frac{9}{p}+3\alpha}[X_{p,\alpha}(R)]^3+C\left(R^{\frac{3}{2p}(1-\theta)+\alpha+\frac{3}{p}(1-\theta)\beta-1}X_{p,\alpha}(R)[Y_{q,\beta}(R)]^{\frac{3}{p}(1-\theta)}\right)^\frac{2p}{3(1-\theta)}\\
&+CR^{2-\frac{3}{p}-\frac{6}{r}+\alpha+2\gamma}X_{p,\alpha}(R)[Z_{r,\gamma}(R)]^2.
\end{align*}
Hence, we get $u,\omega,b\in L^6(\mathbb{R}^3)$ and $\omega,\nabla u,\nabla \omega,\nabla b\in L^2(\mathbb{R}^3)$. The rest of the proof is the same with the case $1\leq q<+\infty$.

\end{proof}

\begin{Rem}\label{Rem3.1}
For the proof of Theorem \ref{main1}, we can easily check that
the inequality $\alpha+2\gamma<\frac{3}{p}+\frac{6}{r}-2$ in $\mathrm{(A2)}$ can be replaced by the equality $\alpha+2\gamma=\frac{3}{p}+\frac{6}{r}-2,$
but the price is that we need to assume in addition that
$$\liminf\limits_{R\rightarrow+\infty}X_{p,\alpha}(R)=0,\;\limsup\limits_{R\rightarrow+\infty}(Y_{q,\beta}(R)+Z_{r,\gamma}(R))<+\infty,\text{ or }$$
$$\limsup\limits_{R\rightarrow+\infty}(X_{p,\alpha}(R)+Y_{q,\beta}(R))<+\infty,\;\liminf\limits_{R\rightarrow+\infty}Z_{r,\gamma}(R)=0.$$
\end{Rem}

\begin{proof}[{\bf Proof of Corollary \ref{Cor1.3}}]
Since $u\in L^p(\mathbb{R}^3)$, $\omega\in L^q(\mathbb{R}^3)$, $b\in L^r(\mathbb{R}^3)$, we have
$$\lim_{R\rightarrow+\infty}\|u\|_{L^p\left(A_R\right)}=\lim_{R\rightarrow+\infty}\|\omega\|_{L^q\left(A_R\right)}=\lim_{R\rightarrow+\infty}\|b\|_{L^r\left(A_R\right)}=0.$$
We divide the range of $p,q,r$ into four cases according to Theorem \ref{main1}.
Applying Theorem \ref{main1} with $\alpha=\beta=\gamma=0$ and observing the endpoint case in Remark \ref{Rem3.1}, we obtain the conclusion.
\end{proof}

The proof of Theorem \ref{main3} is similar to that of Theorem \ref{main1}.
Here we provide another method to establish the energy estimate, which is slightly different. Multiplying both sides of $\eqref{equ1.3}_{1}$ by $u \eta^2-v$, then integrating over $B_t$ and using integration by parts, we obtain
\begin{align}\label{ine5.1}
\aligned
&\int_{B_t}|\nabla u|^2\eta^2dx=\int_{B_t}\nabla u:\nabla vdx+\int_{B_t}\frac{|u|^2}{2}\Delta \eta^2 dx+\int_{B_t}\frac{|u|^2}{2}(u\cdot\nabla\eta^2)dx-\\
&\int_{B_t}(u\cdot \nabla)v\cdot udx+\frac{1}{2} \int_{B_t}\omega\cdot(\eta^2\curl u+\nabla\eta^2\times u)dx-\frac{1}{2} \int_{B_t}\omega\cdot \curl vdx.
\endaligned
\end{align}
Multiplying both sides of $\eqref{equ1.3}_{2}$ by $\omega \eta^2$, then integrating over $B_t$ and using integration by parts, we obtain
\begin{align}\label{ine5.2}
\aligned
&\int_{B_t}|\nabla\omega|^2\eta^2+|\mathrm{div}\omega|^2\eta^2+\kappa|\omega|^2\eta^2dx\\
=&\int_{B_t}\frac{|\omega|^2}{2}\Delta \eta^2+\frac{|\omega|^2}{2}(u\cdot\nabla\eta^2)dx-\int_{B_t}\mathrm{div}\omega(\omega\cdot\nabla\eta^2)dx+\frac{1}{2}\int_{B_t}\curl  u\cdot \omega\eta^2dx.
\endaligned
\end{align}
Collecting \eqref{ine5.1} and \eqref{ine5.2}, we find that
\begin{align}\label{ine5.3}
&\int_{B_t}\left(|\nabla u|^2+|\nabla \omega|^{2}+|\mathrm{div}\omega|^2+\kappa|\omega|^2\right)\eta^2dx\notag\\
=&\int_{B_t}\nabla u:\nabla vdx+\int_{B_t}\curl  u\cdot \omega\eta^2dx-\int_{B_t}\mathrm{div}\omega(\omega\cdot\nabla\eta^2)dx\notag\\
&+\frac{1}{2} \int_{B_t}\omega\cdot(\nabla\eta^2\times u)dx-\frac{1}{2} \int_{B_t}\omega\cdot \curl  vdx+\frac{1}{2}\int_{B_t}(|u|^2+|\omega|^2) \Delta \eta^2dx\\
&+\frac{1}{2} \int_{B_t}(|u|^2+|\omega|^2) (u \cdot \nabla \eta^2) dx-\int_{B_t}(u\otimes u): \nabla v  dx.\notag
\end{align}
From the identity
$$\int_{B_t}|\nabla(u\eta)|^2 d x=\int_{B_t}|\divg (u\eta)|^2 d x+\int_{B_t}|\text{\rm curl}(u\eta)|^2 d x,$$
we see
$$\|\text{\rm curl}(u\eta)\|_{L^2(B_t)}\leq \|\nabla(u\eta)\|_{L^2(B_t)},$$
which implies that
\begin{align}\label{ine5.4}
\|\eta\text{\rm curl}u\|_{L^2(B_t)}&=\|\text{\rm curl}(u\eta)-\nabla\eta\times u\|_{L^2(B_t)}\notag\\
&\leq\|\text{\rm curl}(u\eta)\|_{L^2(B_t)}+\|\nabla\eta\times u\|_{L^2(B_t)}\notag\\
&\leq \|\nabla(u\eta)\|_{L^2(B_t)}+\|\nabla\eta\times u\|_{L^2(B_t)}\\
&\leq \|\eta\nabla u\|_{L^2(B_t)}+C\|u\otimes\nabla\eta\|_{L^2(B_t)}.\notag
\end{align}
Let $\varepsilon$ and $\delta$ be two positive constants to be determined. By the H\"{o}lder inequality, \eqref{ine5.4} and the Young inequality, we have
\begin{align}\label{ine5.5}
&\int_{B_t}\curl  u\cdot \omega\eta^2dx-\int_{B_t}\mathrm{div}\omega(\omega\cdot\nabla\eta^2)dx+\frac{1}{2}\int_{B_t}\omega\cdot(\nabla\eta^2\times u)dx\notag\\
\leq&\|\eta\text{\rm curl}u\|_{L^2(B_t)}\|\omega\eta\|_{L^2(B_t)}-\int_{B_t}\mathrm{div}\omega(\omega\cdot\nabla\eta^2)dx+\|\omega\eta\|_{L^2(B_t)}\|\nabla\eta\times u\|_{L^2(B_t)}\notag\\
\leq&\|\eta\nabla u\|_{L^2(B_t)}\|\omega\eta\|_{L^2(B_t)}-\int_{B_t}\mathrm{div}\omega(\omega\cdot\nabla\eta^2)dx+C\|\omega\eta\|_{L^2(B_t)}\|u\otimes\nabla\eta\|_{L^2(B_t)}\\
\leq&\varepsilon\int_{B_t}|\nabla u|^2\eta^2dx+\frac{1}{4\varepsilon}\int_{B_t}|\omega|^2\eta^2dx+\frac{1}{2}\int_{B_t}|\mathrm{div}\omega|^2\eta^2dx+2\int_{B_t}|\omega\cdot\nabla\eta|^2dx\notag\\
&+\delta\int_{B_t}|\omega|^2\eta^2dx+C_\delta\int_{B_t}|u\otimes\nabla\eta|^2dx.\notag
\end{align}
Here we mention that if we use the point-wise estimate $|\curl  u|\leq \sqrt{2}|\nabla u|$ instead of \eqref{ine5.4}, then we have to require $\kappa$ to be larger.

Combining \eqref{ine5.3} and \eqref{ine5.5}, we can derive that
\begin{align*}
&\int_{B_t}\left[(1-\varepsilon)|\nabla u|^2+|\nabla \omega|^{2}+\frac{1}{2}|\mathrm{div}\omega|^2+\left(\kappa-\frac{1}{4\varepsilon}-\delta\right)|\omega|^2\right]\eta^2dx\\
\leq&\int_{B_t\backslash B_{\frac{\sqrt{3}t}{2}}}\nabla u:\nabla vdx-\frac{1}{2} \int_{B_t\backslash B_{\frac{\sqrt{3}t}{2}}}\omega\cdot \curl  vdx+\frac{C}{(t-s)^2}\int_{B_t \backslash B_s}(|u|^2+|\omega|^2)dx\\
&+\frac{C}{t-s} \int_{B_t\backslash B_s}(|u|^2+|\omega|^2) |u|dx+\int_{B_t\backslash B_{\frac{\sqrt{3}t}{2}}}|u|^2|\nabla v| dx.
\end{align*}
Since $\kappa>\frac{1}{4}$, we can choose some $\varepsilon$ and $\delta$ satisfying
$$\varepsilon\in\left(\frac{1}{4\kappa},1\right),\delta\in\left(0,\kappa-\frac{1}{4\varepsilon}\right).$$
Hence, it follows that
\begin{align*}
&\int_{B_t}\left(|\nabla u|^2+|\nabla \omega|^{2}+|\mathrm{div}\omega|^2+|\omega|^2\right)\eta^2dx\\
\leq&\frac{1}{2}\int_{B_t\backslash  B_{\frac{\sqrt{3}t}{2}}}|\nabla u|^2d x+\frac{1}{4}\int_{B_t\backslash  B_{\frac{\sqrt{3}t}{2}}}|\omega|^2d x+\frac{C}{(t-s)^2}\int_{B_t \backslash B_s}(|u|^2+|\omega|^2)dx\\
&+\frac{C}{t-s} \int_{B_t\backslash B_s}(|u|^2+|\omega|^2) |u|dx+C\int_{B_t\backslash B_{\frac{\sqrt{3}t}{2}}}|u|^2|\nabla v| dx.
\end{align*}
The rest of the proof of Theorem \ref{main3} follows the approach adopted in the proof of Theorem \ref{main1}, and thus is omitted.

\section{Proof of Theorem \ref{main2}}\label{sec4}
In this section, let $\eta$ be a cut-off function defined by
\begin{align*}
\eta(x)= \begin{cases}
1, & |x| <\frac{3R}{2}, \\
4-\frac{2}{R}|x|,& \frac{3R}{2}\leq |x|\leq 2R,\\
0, & |x| >2R.
\end{cases}
\end{align*}
For any $R>0$, we define the function $E(R)$ as follows:
\begin{align}\label{E4.1}
E(R)=\int_{\mathbb{R}^3}\left(|\nabla u|^{2}+|\nabla b|^{2}+|\nabla \omega|^{2}+|\omega|^{2}\right) \eta d x.
\end{align}
We will show some properties of $E(R)$ in the next two lemmas. In Lemma \ref{Lem4.1}, we establish a lower bound estimate for the derivative function $E'(R)$.
In Lemma \ref{Lem4.2}, we establish an upper bound estimate for $E(R)$.
\begin{Lem}\label{Lem4.1}
Let $(u,\pi,\omega,b)$ be a smooth solution of \eqref{equ1.1} and $E(R)$ be defined by \eqref{E4.1}. Then we have
\begin{align}\label{ine4.1}
E'(R)\geq\frac{3}{R}\int_{A_R}\left(|\nabla u|^{2}+|\nabla b|^{2}+|\nabla \omega|^{2}+|\omega|^{2}\right)dx.
\end{align}
\end{Lem}
\begin{proof}
We rewrite $E(R)$ as the following form
\begin{align*}
E(R)=&\int_{B_{\frac{3R}{2}}}\left(|\nabla u|^{2}+|\nabla b|^{2}+|\nabla \omega|^{2}+|\omega|^{2}\right) \eta d x\\
&+\int_{A_R}\left(|\nabla u|^{2}+|\nabla b|^{2}+|\nabla \omega|^{2}+|\omega|^{2}\right)\left(-\frac{2}{R}|x|+4\right)dx.
\end{align*}
By a direct calculation, we obtain
\begin{align*}
E'(R)=&\frac{3}{2}\int_{\partial B_{\frac{3R}{2}}}\left(|\nabla u|^{2}+|\nabla b|^{2}+|\nabla \omega|^{2}+|\omega|^{2}\right) dS\notag\\
&+\int_{A_R}\left(|\nabla u|^{2}+|\nabla b|^{2}+|\nabla \omega|^{2}+|\omega|^{2}\right)\frac{2}{R^2}|x|dx\notag\\
&+2\int_{\partial B_{2R}}\left(|\nabla u|^{2}+|\nabla b|^{2}+|\nabla \omega|^{2}+|\omega|^{2}\right)\left(-\frac{2}{R}\cdot2R+4\right)dS\notag\\
&-\frac{3}{2}\int_{\partial B_{\frac{3R}{2}}}\left(|\nabla u|^{2}+|\nabla b|^{2}+|\nabla \omega|^{2}+|\omega|^{2}\right)\left(-\frac{2}{R}\cdot\frac{3R}{2}+4\right)dS\\
=&\int_{A_R}\left(|\nabla u|^{2}+|\nabla b|^{2}+|\nabla \omega|^{2}+|\omega|^{2}\right)\frac{2}{R^2}|x|dx\notag\\
\geq&\frac{3}{R}\int_{A_R}\left(|\nabla u|^{2}+|\nabla b|^{2}+|\nabla \omega|^{2}+|\omega|^{2}\right)dx.\notag
\end{align*}
\end{proof}

\begin{Lem}\label{Lem4.2}
Let $(u,\pi,\omega,b)$ be a smooth solution of \eqref{equ1.1} and $E(R)$ be defined by \eqref{E4.1}, $\chi\in (0,2)$. Let $\overline{\varphi}_R$ represent the mean value of $\varphi$ on the annulus $A_R$. Denote $U=u-\overline{u}_R$, $B=b-\overline{b}_R$, respectively. Then for any $R\geq1$, it holds that
\begin{align}\label{ine4.18}
E(R)\leq& C\left(\|\nabla U\|_{L^2 (A_R)}^2+\|\nabla B\|_{L^2 (A_R)}^2+\|\nabla\omega\|_{L^2(A_R)}^2+\|\omega\|_{L^2(A_R)}^2\right)+ \notag\\
&CR^{-1}\|u\|_{L^p(A_R)}\|U\|_{L^{2p'}(A_R)}^2+CR^{-1}\|u\|_{L^p(A_R)}\|B\|_{L^{2p'}(A_R)}^2+\\
&CR^{-1}\|u\|_{L^p(A_R)}f(2R)+CR^{\frac{1}{2}-\frac{3}{2p}-\frac{3}{r}}\|u\|_{L^p(A_R)}\|B\|_{L^{2p'}(A_R)}\|b\|_{L^r(A_R)},\notag
\end{align}
where $f$ is defined by \eqref{add-f}.
\end{Lem}
\begin{proof}
Since
$$\int_{A_R}U\cdot\nabla\eta dx=\int_{B_{2R}}U\cdot\nabla\eta dx=\int_{B_{2R}}\divg (U\eta)dx=0,$$
by Lemma \ref{Lem2.2}, there exists $v\in W_{0}^{1,\sigma}(A_R)$ such that $v$ satisfies the following equation
\begin{align*}
\divg v=U\cdot\nabla\eta \text{ in }A_R,
\end{align*}
with the estimate
\begin{align}\label{ine4.2}
\aligned
\|\nabla v\|_{L^\sigma(A_R)}\leq C\|U\cdot\nabla\eta\|_{L^\sigma(A_R)}\leq C R^{-1}\|U\|_{L^\sigma(A_R)},
\endaligned
\end{align}
for any $1<\sigma<+\infty$. We extend $v$ by zero to $B_\frac{3R}{2}$, then $v\in W_{0}^{1,\sigma}(B_{2R}).$

Obviously, $(U,\pi,\omega,B)$ satisfies
\begin{align}\label{equ4.3}
	\left\{
	\begin{array}{ll}
		\text{\rm curl}^2 U+(u\cdot\nabla) U+\nabla \pi=\chi\curl \omega+(b\cdot\nabla) B, \\
		-\Delta \omega+(u\cdot\nabla)\omega-\nabla(\div\omega)+2\chi\omega=\chi\curl U,\\
   -\Delta B+(u\cdot\nabla)B-(b\cdot\nabla)U=0,\\
		\divg U=\divg B=0.
	\end{array}
	\right.
\end{align}
Denote the $i$-th component of $U$, $\omega$ and $B$ by $U_i$, $\omega_i$ and $B_i$, respectively. Multiply both sides of $\eqref{equ4.3}_{1}$, $\eqref{equ4.3}_{2}$ and $\eqref{equ4.3}_{3}$ by $U \eta-w$, $\omega \eta$ and $B \eta$ respectively, integrate over $B_{2R}$ and apply integration by parts. This procedure yields
\begin{align}\label{ine4.4}
&\int_{B_{2R}}\left(|\curl U|^{2}+|\nabla \omega|^{2}+|\div \omega|^{2}+2\chi|\omega|^{2}+|\nabla B|^{2}\right)\eta d x\notag\\
=&-\int_{B_{2R}}\left[\curl U\cdot(\nabla\eta\times U)+(\omega\cdot\nabla\eta)\div\omega+\sum_{i=1}^3 \big[\nabla \omega_i\cdot (\omega_i\nabla\eta)+\nabla B_i\cdot (B_i\nabla\eta)\big]\right]d x\notag\\
&+\frac{1}{2} \int_{B_{2R}}(|U|^2+|B|^2)u \cdot \nabla \eta d x+\frac{1}{2} \int_{B_{2R}}|\omega|^2u \cdot \nabla \eta d x+\int_{B_{2R}}\curl U\cdot\curl v dx\\
&+2\chi\int_{B_{2R}}\curl U\cdot\omega\eta d x+\chi\int_{B_{2R}}\left[\omega\cdot(\nabla\eta\times U)-\omega\cdot\curl v\right] d x\notag\\
&- \int_{B_{2R}}(u \cdot\nabla )v \cdot U dx+\int_{B_{2R}}(b \cdot\nabla )v\cdot B dx-\int_{B_{2R}}(U \cdot B)b\cdot\nabla \eta dx\notag\\
:=&\sum_{i=1}^{9}I_i.\notag
\end{align}
With the help of the H\"{o}lder inequaity and the Poincar\'{e} inequality
$$\|\varphi-\overline{\varphi}_R\|_{L^2(A_R)}\leq CR\|\nabla \varphi\|_{L^2(A_R)}\text{ for any } \varphi\in W^{1,2}(A_R),$$
we have
\begin{align}\label{ine4.5}
I_1\leq& CR^{-1}\|\curl U\|_{L^2 (A_R)}\|U\|_{L^2 (A_R)}+CR^{-1}\|\divg \omega\|_{L^2 (A_R)}\|\omega\|_{L^2 (A_R)}\notag\\
&+CR^{-1}\|\nabla \omega\|_{L^2 (A_R)}\|\omega\|_{L^2 (A_R)}+CR^{-1}\|\nabla B\|_{L^2 (A_R)}\|B\|_{L^2 (A_R)}\notag\\
\leq& C\left(\|\nabla U\|_{L^2 (A_R)}^2+\|\nabla B\|_{L^2 (A_R)}^2\right)+CR^{-1}\|\nabla \omega\|_{L^2 (A_R)}\|\omega\|_{L^2 (A_R)}\\
\leq&  C\left(\|\nabla U\|_{L^2 (A_R)}^2+\|\nabla B\|_{L^2 (A_R)}^2+\|\nabla \omega\|_{L^2 (A_R)}^2+\|\omega\|_{L^2 (A_R)}^2\right),\notag
\end{align}
where we require $R\geq1$ in the last step.

By the H\"{o}lder inequality, we have
\begin{align}\label{ine4.8}
I_2&\leq CR^{-1}\left(\|u\|_{L^p(A_R)}\|U\|_{L^{2p'}(A_R)}^2+\|u\|_{L^p(A_R)}\|B\|_{L^{2p'}(A_R)}^2\right),
\end{align}
and
\begin{align}\label{add-w}
|\overline{\omega}_R|&\leq CR^{-\frac{3}{2}}\|\omega\|_{L^2(A_R)}.
\end{align}
Combining the H\"{o}lder inequality, the interpolation inequality, the Minkowski inequality, the following standard inequality
$$(a_1+a_2)^t\leq \max\{2^{t-1},1\}(a_1^t+a_2^t),\text{ where }a_1,a_2\geq0,\;t>0,$$
the Sobolev-Poincar\'{e} inequality
$$\|\varphi-\overline{\varphi}_R\|_{L^6(A_R)}\leq C\|\nabla \varphi\|_{L^2(A_R)}\text{ for any } \varphi\in W^{1,2}(A_R),$$
and \eqref{add-w}, we deduce
\begin{align}\label{ine4.9}
I_3&\leq CR^{-1}\|u\|_{L^p(A_R)}\|\omega\|_{L^{2p'}(A_R)}^2\notag\\
&\leq CR^{-1}\|u\|_{L^p(A_R)}\|\omega\|_{L^{2}(A_R)}^{2-\frac{3}{p}}\|\omega\|_{L^{6}(A_R)}^{\frac{3}{p}}\notag\\
&\leq CR^{-1}\|u\|_{L^p(A_R)}\|\omega\|_{L^{2}(A_R)}^{2-\frac{3}{p}}\left(\|\omega-\overline{\omega}_R\|_{L^{6}(A_R)}+\|\overline{\omega}_R\|_{L^{6}(A_R)}\right)^{\frac{3}{p}}\notag\\
&\leq CR^{-1}\|u\|_{L^p(A_R)}\|\omega\|_{L^{2}(A_R)}^{2-\frac{3}{p}}\left(\|\omega-\overline{\omega}_R\|_{L^{6}(A_R)}^{\frac{3}{p}}
+\|\overline{\omega}_R\|_{L^{6}(A_R)}^{\frac{3}{p}}\right)\\
&\leq CR^{-1}\|u\|_{L^p(A_R)}\|\omega\|_{L^{2}(A_R)}^{2-\frac{3}{p}}\left(\|\nabla\omega\|_{L^2(A_R)}^{\frac{3}{p}}
+R^{-\frac{3}{p}}\|\omega\|_{L^{2}(A_R)}^{\frac{3}{p}}\right)\notag\\
&\leq CR^{-1}\|u\|_{L^p(A_R)}\cdot f(2R),\notag
\end{align}
where the function $f$ is defined by \eqref{add-f}, and we require $R\geq1$ in the last step.

Using the H\"{o}lder inequality, \eqref{ine4.2} and the Poincar\'{e} inequality, we get
\begin{align}\label{ine4.10}
I_4
&\leq C\|\nabla U\|_{L^2(A_R)}\|\nabla v\|_{L^2(A_R)}\notag\\
&\leq CR^{-1}\|\nabla U\|_{L^2(A_R)}\|U\|_{L^2(A_R)}\\
&\leq C\|\nabla U\|_{L^2(A_R)}^2.\notag
\end{align}
Using the Young inequality, we obtain
\begin{align}\label{ine4.11}
\aligned
I_5&\leq 2\chi\int_{B_{2R}}|\curl U|\eta^{\frac{1}{2}}\cdot|\omega|\eta^{\frac{1}{2}} d x\\
&\leq \varepsilon\int_{B_{2R}}|\curl U|^2\eta d x+\frac{\chi^2}{\varepsilon}\int_{B_{2R}}|\omega|^2\eta d x.
\endaligned
\end{align}
Using the H\"{o}lder inequality, \eqref{ine4.2} and the Poincar\'{e} inequality, we get
\begin{align}\label{ine4.12}
I_6&\leq C\chi\left(\|\omega\|_{L^2(A_R)}\|\nabla\eta\times U\|_{L^{2}(A_R)}+\|\omega\|_{L^2(A_R)}\|\curl v\|_{L^{2}(A_R)}\right)\notag\\
&\leq C\chi\left(R^{-1}\|\omega\|_{L^2(A_R)}\| U\|_{L^{2}(A_R)}+\|\omega\|_{L^2(A_R)}\|\nabla v\|_{L^{2}(A_R)}\right)\notag\\
&\leq C\chi R^{-1}\|\omega\|_{L^2(A_R)}\| U\|_{L^{2}(A_R)}\\
&\leq C\chi \|\omega\|_{L^2(A_R)}\| \nabla U\|_{L^{2}(A_R)}\notag\\
&\leq C\chi\left(\| \nabla U\|_{L^{2}(A_R)}^2+\|\omega\|_{L^2(A_R)}^2\right).\notag
\end{align}
By the H\"{o}lder inequality and \eqref{ine4.2}, we have
\begin{align}\label{ine4.13}
\aligned
I_7&\leq \|u\|_{L^p(A_R)}\|\nabla v\|_{L^{2p'}(A_R)}\|U\|_{L^{2p'}(A_R)}\\
&\leq CR^{-1}\|u\|_{L^p(A_R)}\|U\|_{L^{2p'}(A_R)}^2.
\endaligned
\end{align}
Using the H\"{o}lder inequality, \eqref{ine4.2} and the Minkowski inequality, we get
\begin{align}\label{ine4.14}
I_8+I_9&\leq \left(\|\nabla v\|_{L^p(A_R)}+CR^{-1}\|U\|_{L^p(A_R)}\right)\|B\|_{L^{2p'}(A_R)}\|b\|_{L^{2p'}(A_R)}\notag\\
&\leq CR^{-1}\|U\|_{L^p(A_R)}\|B\|_{L^{2p'}(A_R)}\|b\|_{L^{2p'}(A_R)}\notag\\
&\leq CR^{-1}\|U\|_{L^p(A_R)}\|B\|_{L^{2p'}(A_R)}\left(\|B\|_{L^{2p'}(A_R)}+\left\|\overline{b}_R\right\|_{L^{2p'}(A_R)}\right)\\
&\leq CR^{-1}\|U\|_{L^p(A_R)}\|B\|_{L^{2p'}(A_R)}\left(\|B\|_{L^{2p'}(A_R)}+CR^{\frac{3}{2}-\frac{3}{2p}-\frac{3}{r}}\|b\|_{L^r(A_R)}\right)\notag\\
&\leq CR^{-1}\|u\|_{L^p(A_R)}\|B\|_{L^{2p'}(A_R)}\left(\|B\|_{L^{2p'}(A_R)}+CR^{\frac{3}{2}-\frac{3}{2p}-\frac{3}{r}}\|b\|_{L^r(A_R)}\right).\notag
\end{align}

Plugging the estimates of $I_1-I_9$ into \eqref{ine4.4}, we conclude
\begin{align}\label{ine4.15}
&\int_{B_{2R}}\left[(1-\varepsilon)|\curl U|^{2}+|\nabla \omega|^{2}+|\div \omega|^{2}+\left(2\chi-\frac{\chi^2}{\varepsilon}\right)|\omega|^{2}+|\nabla B|^{2}\right]\eta d x\notag\\
\leq& C\left(\|\nabla U\|_{L^2 (A_R)}^2+\|\nabla B\|_{L^2 (A_R)}^2+\|\nabla\omega\|_{L^2(A_R)}^2+\|\omega\|_{L^2(A_R)}^2\right)+ \notag\\
&CR^{-1}\|u\|_{L^p(A_R)}\|U\|_{L^{2p'}(A_R)}^2+CR^{-1}\|u\|_{L^p(A_R)}\|B\|_{L^{2p'}(A_R)}^2+\\
&CR^{-1}\|u\|_{L^p(A_R)}f(2R)+CR^{\frac{1}{2}-\frac{3}{2p}-\frac{3}{r}}\|u\|_{L^p(A_R)}\|B\|_{L^{2p'}(A_R)}\|b\|_{L^r(A_R)}.\notag
\end{align}
Since $\chi\in(0,2)$, we can choose some $\varepsilon\in\left(\frac{\chi}{2},1\right)$ such that
$$1-\varepsilon>0 \text{ and }2\chi-\frac{\chi^2}{\varepsilon}>0.$$

Multiply both sides of the identity $\curl ^2U=-\Delta U$ by $U \eta$, integrate over $B_{2R}$ and apply integration by parts. This procedure yields
\begin{align}\label{ine4.16}
\int_{B_{2R}}|\nabla U|^2\eta dx=\int_{B_{2R}}\left(|\curl  U|^2\eta+\curl  U\cdot(\nabla\eta\times U)-\sum_{i=1}^3\nabla U_i\cdot(U_i\nabla\eta)\right)dx.
\end{align}
Applying the H\"{o}lder inequality and the Poincar\'{e} inequality to the above equality, we get
\begin{align}\label{ine4.17}
\int_{B_{2R}}|\nabla U|^2\eta dx\leq&\int_{B_{2R}}|\curl  U|^2\eta dx+\|\curl  U\|_{L^2(A_R)}\|\nabla\eta\times U\|_{L^{2}(A_R)}\notag\\
&+\sum_{i=1}^3\|\nabla U_i\|_{L^2(A_R)}\| U_i\nabla\eta\|_{L^{2}(A_R)}\\
\leq&\int_{B_{2R}}|\curl  U|^2\eta dx+CR^{-1}\|\nabla U\|_{L^2(A_R)}\|U\|_{L^{2}(A_R)}\notag\\
\leq&\int_{B_{2R}}|\curl  U|^2\eta dx+C\|\nabla U\|_{L^2(A_R)}^2.\notag
\end{align}
Combining \eqref{ine4.15} and \eqref{ine4.17}, and recalling the definition of $E(R)$, we conclude that \eqref{ine4.18} holds.
\end{proof}

In order to prove Theorem \ref{main2}, we also need the estimate for $f(R)$.
\begin{Lem}
Let $(u,\pi,\omega,b)$ be a smooth solution of \eqref{equ1.1}. Suppose $q\in[1,+\infty]$, $r\in [1,6]$, $\beta\in\left[0,+\infty\right)$, $\gamma\in\left[0,\frac{3}{r}-\frac{1}{2}\right]$, $\lambda,\nu\geq0$, $\chi\in(0,2)$. Assume that one of the assumptions $\mathrm{(B1)}$, $\mathrm{(B2)}$, $\mathrm{(B3)}$ in Theorem \ref{main2} holds.
Then there exist three positive constants $R_2>3$, $A$ and $C$ such that
\begin{align}\label{ine4.21}
f(R)\leq C(\ln R)^{A},\;\forall R>R_2.
\end{align}
\end{Lem}

\begin{proof}
Based on the assumptions (B1), (B2) and (B3), we infer that there exist two positive constants $R_1>3$ and $C$
such that the following three inequalities hold for any $R>R_1$:
\begin{equation}\label{ine4.19}
\|u\|_{L^p\left(A_R\right)}\leq CR^\alpha(\ln R)^\lambda,
\end{equation}
\begin{equation}\label{add-w2}
\|\omega\|_{L^q\left(A_R\right)}\leq CR^\beta,
\end{equation}
\begin{equation}\label{ine4.20}
\|b\|_{L^r\left(A_R\right)}\leq CR^\gamma(\ln R)^\nu.
\end{equation}
Here we make a convention that when $p\in[3,\frac{9}{2}]$, we take $\lambda=0$.

For large $R$, we demonstrate $f(R)$ possesses the following estimates:
\begin{itemize}
\item[(i)] When $\frac{3}{2}< p<3$ and $1 \leq r<2p'$, we have $f(R)\leq C(\ln R)^{\frac{3p\lambda}{2p-3}}+C(\ln R)^{\frac{(6-r)p'\lambda}{(3-p')r}+2\nu}$.
\item[(ii)] When $\frac{3}{2}<p<3$ and $2p' \leq r\leq6$, we have $f(R)\leq C(\ln R)^{\frac{3p\lambda}{2p-3}}+C(\ln R)^{2\nu}$.
\item[(iii)] When $3\leq p\leq\frac{9}{2}$ and $1 \leq r<2p'$, we have $f(R)\leq C(\ln R)^{2\nu}$.
\end{itemize}
Since the above three items are similar, we only illustrate the first item.
When $\frac{3}{2}< p<3$, $1 \leq r<2p'$ and $1\leq q<+\infty$,
substituting \eqref{ine4.19}, \eqref{add-w2} and \eqref{ine4.20} into \eqref{add1}, we obtain
\begin{align}\label{add-f4}
f(R)\leq&CR^{1-\frac{6}{p}+2\alpha}(\ln R)^{2\lambda}+CR^{1-\frac{6}{r}+2\gamma}(\ln R)^{2\nu}+CR^{\frac{p-6}{2p-3}+\frac{3p}{2p-3}\alpha}(\ln R)^\frac{3p\lambda}{2p-3}\notag\\
&+C\left(h_1(q)R^{-\frac{6-q}{q(1-\theta)}+2\beta}+R^{3-\frac{6}{q}-\frac{2}{1-\theta}+2\beta}\right)\notag\\
&+Ch_2(q)\left(R^{\alpha+\frac{(6-2p')q}{(6-q)p'}(1-\theta)\beta-1}(\ln R)^\lambda\right)^{\frac{(6-q)p'}{(3-p')q(1-\theta)}}\\
&+Ch_3(q)\left(R^{\alpha+\frac{2q}{(q-2)p}(1-\theta)\beta-1}(\ln R)^\lambda\right)^\frac{(q-2)p}{(1-\theta)q}\notag\\
&+CR^{-\frac{(6-r)p'}{(3-p')r}+\frac{(6-r)p'}{(3-p')r}\alpha+2\gamma}(\ln R)^{\frac{(6-r)p'\lambda}{(3-p')r}+2\nu}\notag\\
\leq&C(\ln R)^{\frac{3p\lambda}{2p-3}}+C(\ln R)^{\frac{(6-r)p'\lambda}{(3-p')r}+2\nu},\notag
\end{align}
where we have used the facts that we state in the lines between \eqref{add1} and \eqref{ine3.13}.
When $\frac{3}{2}< p<3$, $1 \leq r<2p'$ and $q=+\infty$,
substituting \eqref{ine4.19}, \eqref{add-w2} and \eqref{ine4.20} into \eqref{add-f1}, we conclude
\begin{align}\label{add-f5}
f(R)\leq&CR^{1-\frac{6}{p}+2\alpha}(\ln R)^{2\lambda}+CR^{1-\frac{6}{r}+2\gamma}(\ln R)^{2\nu}\notag\\
&+CR^{3-\frac{2}{1-\theta}+2\beta}+CR^{\frac{p-6}{2p-3}+\frac{3p}{2p-3}\alpha}(\ln R)^\frac{3p\lambda}{2p-3}\notag\\
&+C\left(R^{\frac{3}{2p}(1-\theta)+\alpha+\frac{3}{p}(1-\theta)\beta-1}(\ln R)^\lambda\right)^\frac{2p}{3(1-\theta)}\\
&+CR^{-\frac{(6-r)p'}{(3-p')r}+\frac{(6-r)p'}{(3-p')r}\alpha+2\gamma}(\ln R)^{\frac{(6-r)p'\lambda}{(3-p')r}+2\nu}\notag\\
\leq&C(\ln R)^{\frac{3p\lambda}{2p-3}}+C(\ln R)^{\frac{(6-r)p'\lambda}{(3-p')r}+2\nu}.\notag
\end{align}
Putting \eqref{add-f4} and \eqref{add-f5} together, we complete the proof of the first item.

We choose
\begin{align*}
A=\max\left\{{\frac{3p\lambda}{2p-3},\;\frac{(6-r)p'\lambda}{(3-p')r}+2\nu},\;1\right\},
\end{align*}
and then we can get a unified estimate for $f(R)$:
\begin{align*}
f(R)\leq C(\ln R)^{A},\;\forall R>R_2,
\end{align*}
where $R_2$ is some constant satisfying $R_2>R_1$.
\end{proof}

With the above preparations, we are ready to prove Theorem \ref{main2}.

\begin{proof}[{\bf Proof of Theorem \ref{main2}}]
For convenience, we denote the five terms in the right hand side of \eqref{ine4.18} by $K_1,K_2,K_3,K_4,K_5$, respectively.
Denote
$$\theta=\max\left\{\frac{A}{A+1},\;\frac{9-3p}{6-p},\;\frac{6p-3r(p-1)}{[6-rh_2(r)]p}h_2(r)\right\},$$
where $h_2$ is defined in Lemma \ref{Lemma2.6}. It is not difficult to verify the fact that $\theta\in(0,1)$.
Using \eqref{ine4.1} and \eqref{ine4.21}, we have
\begin{align}\label{ine4.22}
K_1=K_1^\theta K_1^{1-\theta}\leq&\left[f(2R)\right]^{1-\theta}\left[\frac{1}{3}R E'(R)\right]^{\theta}\notag\\
\leq& C(\ln R)^{A(1-\theta)}\left[R E'(R)\right]^{\theta}\\
\leq&C\left[R\ln R E'(R)\right]^{\theta},\notag
\end{align}
where we require $A(1-\theta)\leq\theta$ and $R>R_2$. Using \eqref{ine4.21} and \eqref{ine4.19}, we see
 $\displaystyle\lim_{R\rightarrow+\infty}K_4=0$.

We claim that $E(R)\equiv0$, otherwise, in view of the nondecreasing property of $E(R)$, there exists a constant $R_3$ satisfying $R_3>R_2$ such that
$$\text{$E(R)\geq E(R_3)>0$ for any $R\geq R_3$.}$$
Since $\displaystyle\lim_{R\rightarrow+\infty}K_4=0$, there exists a constant $R_4> R_3$ such that
\begin{align}\label{ine4.23}
K_4\leq\frac{1}{8}E(R_3)\leq\frac{1}{8}E(R),\;\forall R>R_4.
\end{align}
Using the Minkowski inequality and the H\"{o}lder inequality, we derive
\begin{align}\label{ine4.6}
\|U\|_{L^p(A_R)}&\leq \|u\|_{L^p(A_R)}+\|\overline{u}_R\|_{L^p(A_R)}\notag\\
&\leq\|u\|_{L^p(A_R)}+CR^\frac{3}{p}|\overline{u}_R|\\
&\leq C\|u\|_{L^p(A_R)}.\notag
\end{align}
Similarly, we can derive
\begin{align}\label{ine4.7}
\|B\|_{L^r(A_R)}&\leq C\|b\|_{L^r(A_R)}.
\end{align}

\textbf{Assume that (B1) holds.}
Using the interpolation inequality, \eqref{ine4.6} and the Sobolev-Poincar\'{e} inequality, we obtain
\begin{align}\label{ine4.24}
K_2&\leq CR^{-1}\|u\|_{L^p(A_R)}\left(\|U\|_{L^p(A_R)}^{\frac{2p-3}{6-p}}\|U\|_{L^6(A_R)}^{\frac{9-3p}{6-p}}\right)^2 \notag\\
&\leq CR^{-1}\|u\|_{L^p(A_R)}^{\frac{3p}{6-p}}\|\nabla U\|_{L^2(A_R)}^{\frac{2(9-3p)}{6-p}}\\
&\leq CR^{-1+\frac{3p\alpha}{6-p}}(\ln R)^{\frac{3p\lambda}{6-p}}[RE'(R)]^{\frac{9-3p}{6-p}}\notag\\
&\leq C\left[R\ln RE'(R)\right]^{\frac{9-3p}{6-p}}.\notag
\end{align}
Here we have used the conditions
$$\alpha\leq \frac{2}{p}-\frac{1}{3} \text{ and } \lambda\leq\frac{3}{p}-1.$$
By the interpolation inequality, \eqref{ine4.7} and the Sobolev-Poincar\'{e} inequality, we obtain
\begin{align}\label{ine4.25}
K_3&\leq CR^{-1}\|u\|_{L^p(A_R)}\left(\|B\|_{L^r(A_R)}^{\frac{(3-p')r}{(6-r)p'}}\|B\|_{L^6(A_R)}^{\frac{6p'-3r}{(6-r)p'}}\right)^2 \notag\\
&\leq CR^{-1}\|u\|_{L^p(A_R)}\|b\|_{L^r(A_R)}^{\frac{2(3-p')r}{(6-r)p'}}\|\nabla B\|_{L^2(A_R)}^{\frac{12p'-6r}{(6-r)p'}}\notag\\
&\leq CR^{-1+\alpha+\frac{2(3-p')r}{(6-r)p'}\gamma}(\ln R)^{\lambda+\frac{2(3-p')r}{(6-r)p'}\nu}[RE'( R)]^{\frac{6p'-3r}{(6-r)p'}}\\
&\leq C\left[R\ln RE'(R)\right]^{\frac{6p'-3r}{(6-r)p'}}\notag\\
&= C\left[R\ln RE'(R)\right]^{\frac{6p-3r(p-1)}{(6-r)p}},\notag
\end{align}
where we have used the conditions
$$\alpha+\frac{(4p-6)r}{(6-r)p}\gamma\leq1\text{ and }\lambda+\frac{(4p-6)r}{(6-r)p}\nu\leq\frac{6p-3r(p-1)}{(6-r)p}.$$

By the interpolation inequality, \eqref{ine4.7} and the Sobolev-Poincar\'{e} inequality, we obtain
\begin{align}\label{ine4.26}
K_5&\leq CR^{\frac{1}{2}-\frac{3}{2p}-\frac{3}{r}}\|u\|_{L^p(A_R)}\|B\|_{L^r(A_R)}^{\frac{(3-p')r}{(6-r)p'}}\|B\|_{L^6(A_R)}^{\frac{6p'-3r}{(6-r)p'}} \|b\|_{L^r(A_R)}\notag\\
&\leq CR^{\frac{1}{2}-\frac{3}{2p}-\frac{3}{r}}\|u\|_{L^p(A_R)}\|b\|_{L^r(A_R)}^{1+\frac{(3-p')r}{(6-r)p'}}\|\nabla B\|_{L^2(A_R)}^{\frac{6p'-3r}{(6-r)p'}} \\
&\leq CR^{\frac{1}{2}-\frac{3}{2p}-\frac{3}{r}+\alpha+\left[1+\frac{(3-p')r}{(6-r)p'}\right]\gamma}(\ln R)^{\lambda+\left[1+\frac{(3-p')r}{(6-r)p'}\right]\nu}[RE'( R)]^{\frac{6p'-3r}{2(6-r)p'}}.\notag
\end{align}
Notice that
\begin{align*}
&\frac{1}{2}-\frac{3}{2p}-\frac{3}{r}+\alpha+\left[1+\frac{(3-p')r}{(6-r)p'}\right]\gamma\notag\\
\leq&\frac{1}{2}-\frac{3}{2p}-\frac{3}{r}+1-\frac{(4p-6)r}{(6-r)p}\gamma+\left[1+\frac{(3-p')r}{(6-r)p'}\right]\gamma\notag\\
=&\frac{3}{2}-\frac{3}{2p}-\frac{3}{r}+\frac{6p-3pr+3r}{(6-r)p}\gamma\\
\leq&\frac{3}{2}-\frac{3}{2p}-\frac{3}{r}+\frac{6p-3pr+3r}{(6-r)p}\left(\frac{3}{r}-\frac{1}{2}\right)\notag\\
=&0.\notag
\end{align*}
Moreover, we claim that if $(\alpha,\gamma)\neq\left(\frac{3}{p}-1,\frac{3}{r}-\frac{1}{2}\right)$, then the index
$\frac{1}{2}-\frac{3}{2p}-\frac{3}{r}+\alpha+\left[1+\frac{(3-p')r}{(6-r)p'}\right]\gamma$ can not attain zero, otherwise, we have
\begin{align*}
\begin{cases}
\gamma=\frac{3}{r}-\frac{1}{2}\\
\alpha=1-\frac{(4p-6)r}{(6-r)p}\gamma=\frac{3}{p}-1,
\end{cases}
\end{align*}
which is impossible.
Hence, when $(\alpha,\gamma)\neq\left(\frac{3}{p}-1,\frac{3}{r}-\frac{1}{2}\right)$, there exists a constant $R_5> R_4$ such that
\begin{align*}
K_5&\leq C\left[RE'(R)\right]^{\frac{6p-3r(p-1)}{2(6-r)p}}\leq C\left[R\ln RE'(R)\right]^{\frac{6p-3r(p-1)}{2(6-r)p}},\;\forall R\geq R_5.
\end{align*}
When $(\alpha,\gamma)=\left(\frac{3}{p}-1,\frac{3}{r}-\frac{1}{2}\right)$, the index
$\frac{1}{2}-\frac{3}{2p}-\frac{3}{r}+\alpha+\left[1+\frac{(3-p')r}{(6-r)p'}\right]\gamma=0$,
so we need to require
$$\lambda+\left[1+\frac{(3-p')r}{(6-r)p'}\right]\nu\leq\frac{6p'-3r}{2(6-r)p'},\text{ i.e. }\lambda+\frac{6p+pr-3r}{(6-r)p}\nu\leq\frac{6p-3r(p-1)}{2(6-r)p},$$
which guarantees
\begin{align*}
K_5
&\leq C\left[R\ln RE'(R)\right]^{\frac{6p'-3r}{2(6-r)p}}= C\left[R\ln RE'(R)\right]^{\frac{6p-3r(p-1)}{2(6-r)p}}.
\end{align*}
Therefore, no matter $(\alpha,\gamma)$  equals to $\left(\frac{3}{p}-1,\frac{3}{r}-\frac{1}{2}\right)$ or not, it always holds that
\begin{align}\label{ine4.27}
K_5\leq C\left[R\ln RE'(R)\right]^{\frac{6p-3r(p-1)}{2(6-r)p}},\;\forall R\geq R_5.
\end{align}

Combining \eqref{ine4.18}, \eqref{ine4.22}, \eqref{ine4.23}, \eqref{ine4.24}, \eqref{ine4.25} and \eqref{ine4.27}, we obtain
\begin{align}\label{ine4.28}
E(R)\leq& C\left[R\ln R E'(R)\right]^{\theta}+C\left[R\ln RE'(R)\right]^{\frac{9-3p}{6-p}}+C\left[R\ln RE'(R)\right]^{\frac{6p-3r(p-1)}{(6-r)p}}\notag\\
&+\frac{1}{8}E(R)+C\left[R\ln RE'(R)\right]^{\frac{6p-3r(p-1)}{2(6-r)p}},\;\forall R\geq R_5.
\end{align}
Applying the Young inequality, we have
\begin{align}\label{ine4.29}
C\left[R\ln RE'(R)\right]^{\frac{9-3p}{6-p}}\leq\frac{1}{8}E(R_3)+C\left[R\ln RE'(R)\right]^{\theta},
\end{align}
\begin{align}\label{ine4.30}
C\left[R\ln RE'(R)\right]^{\frac{6p-3r(p-1)}{(6-r)p}}\leq\frac{1}{8}E(R_3)+C\left[R\ln RE'(R)\right]^{\theta},
\end{align}
\begin{align}\label{ine4.31}
C\left[R\ln RE'(R)\right]^{\frac{6p-3r(p-1)}{2(6-r)p}}\leq\frac{1}{8}E(R_3)+C\left[R\ln RE'(R)\right]^{\theta}.
\end{align}
Combining \eqref{ine4.28}, \eqref{ine4.29}, \eqref{ine4.30} and \eqref{ine4.31}, we obtain
\begin{align*}
E(R)&\leq C\left[R\ln RE'(R)\right]^{\theta}+\frac{1}{8}E(R)+\frac{3}{8}E(R_3)\notag\\
&\leq C\left[R\ln RE'(R)\right]^{\theta}+\frac{1}{2}E(R),\;\forall R\geq R_5,
\end{align*}
which implies
\begin{align*}
E(R)
&\leq C\left[R\ln RE'(R)\right]^{\theta},\;\forall R\geq R_5.
\end{align*}
Consequently, it follows that
\begin{align*}
\ln\ln R-\ln\ln R_5=\int_{R_5}^{R}\frac{1}{\rho\ln\rho}d\rho\leq\int_{R_5}^{R}\frac{CE'(\rho)}{E(\rho)^{\frac{1}{\theta}}}d\rho\leq CE(R_5)^{1-\frac{1}{\theta}}<+\infty.
\end{align*}
Letting $R\rightarrow+\infty$, the above inequality leads to a contradiction. Thus, $E(R)\equiv0$.

Thanks to the simple inequality
$$\|\nabla u\|_{L^2(B_R)}^2+\|\nabla b\|_{L^2(B_R)}^2+\|\nabla \omega\|_{L^2(B_R)}^2+\|\omega\|_{L^2(B_R)}^2\leq E(R),$$
we conclude that $u,b$ are constant vectors and $\omega=0$. Finally, the condition
$$\limsup\limits_{R\rightarrow+\infty}\left[X_{p,\alpha,\lambda}(R)+Y_{q,\beta}(R)+Z_{r,\gamma,\nu}(R)\right]<+\infty$$
forces $u$ and $b$ to be zero.
\par \textbf{Assume that (B2) holds.} Using the H\"{o}lder inequality and \eqref{ine4.7}, we obtain
\begin{align}\label{ine4.32}
\aligned
K_3&\leq CR^{-1}\|u\|_{L^p(A_R)}\left(\|B\|_{L^r(A_R)}CR^{3\left(\frac{1}{2p'}-\frac{1}{r}\right)}\right)^2\\
&\leq CR^{2-\frac{3}{p}-\frac{6}{r}}\|u\|_{L^p(A_R)}\|b\|_{L^r(A_R)}^{2},
\endaligned
\end{align}
and
\begin{align}\label{ine4.33}
\aligned
K_5&\leq CR^{\frac{1}{2}-\frac{3}{2p}-\frac{3}{r}}\|u\|_{L^p(A_R)}\|B\|_{L^r(A_R)}CR^{3(\frac{1}{2p'}-\frac{1}{r})}\|b\|_{L^r(A_R)}\\
&\leq CR^{2-\frac{3}{p}-\frac{6}{r}}\|u\|_{L^p(A_R)}\|b\|_{L^r(A_R)}^{2}.
\endaligned
\end{align}
Combining \eqref{ine4.18}, \eqref{ine4.22}, \eqref{ine4.23}, \eqref{ine4.24}, \eqref{ine4.32} and \eqref{ine4.33}, we obtain
\begin{align}\label{ine4.34}
E(R)\leq& C\left[R\ln R E'(R)\right]^{\theta}+C\left[R\ln RE'(R)\right]^{\frac{9-3p}{6-p}}\notag\\
&+\frac{1}{8}E(R)+CR^{2-\frac{3}{p}-\frac{6}{r}}\|u\|_{L^p(A_R)}\|b\|_{L^r(A_R)}^2.
\end{align}
Since $\alpha+2\gamma<\frac{3}{p}+\frac{6}{r}-2$, it holds that
\begin{align*}
\lim_{R\rightarrow+\infty}CR^{2-\frac{3}{p}-\frac{6}{r}}\|u\|_{L^p(A_R)}\|b\|_{L^r(A_R)}^2=0.
\end{align*}
As a consequence, there exists a constant $R_6>R_3$ such that
\begin{align}\label{ine4.35}
CR^{2-\frac{3}{p}-\frac{6}{r}}\|u\|_{L^p}\|b\|_{L^r}^2\leq\frac{1}{8}E(R_3)\leq\frac{1}{8}E(R),\; \forall R\geq R_6.
\end{align}
From \eqref{ine4.34}, \eqref{ine4.29} and \eqref{ine4.35}, we have
\begin{align*}
E(R)\leq& C\left[R\ln R E'(R)\right]^{\theta}+\frac{3}{8}E(R),
\end{align*}
which yields
\begin{align*}
E(R)
&\leq C\left[R\ln RE'(R)\right]^{\theta}.
\end{align*}
Routinely, we have $u=b=\omega=0$.
\par \textbf{Assume that (B3) holds.} Using the H\"{o}lder inequality and \eqref{ine4.6}, we obtain
\begin{align}\label{ine4.36}
\aligned
K_2&\leq CR^{-1}\|u\|_{L^p(A_R)}\left(\|U\|_{L^p(A_R)}CR^{3\left(\frac{1}{2p'}-\frac{1}{p}\right)}\right)^2\leq CR^{2-\frac{9}{p}}\|u\|_{L^p(A_R)}^3.
\endaligned
\end{align}
Similarly to the case that (B1) holds, we can establish the estimate for $K_5$:
\begin{align}\label{ine4.37}
K_5&\leq C\left[R\ln RE'(R)\right]^{\frac{6p-3r(p-1)}{2(6-r)p}}.
\end{align}
Combining \eqref{ine4.18}, \eqref{ine4.22}, \eqref{ine4.23}, \eqref{ine4.25}, \eqref{ine4.36} and \eqref{ine4.37}, we obtain
\begin{align}\label{ine4.38}
\aligned
E(R)\leq& C\left[R\ln R E'(R)\right]^{\theta}+CR^{2-\frac{9}{p}}\|u\|_{L^p(A_R)}^3+C\left[R\ln RE'(R)\right]^{\frac{6p-3r(p-1)}{(6-r)p}}\\
&+\frac{1}{8}E(R)+C\left[R\ln RE'(R)\right]^{\frac{6p-3r(p-1)}{2(6-r)p}}.
\endaligned
\end{align}
Since
\begin{align*}
\lim_{R\rightarrow+\infty} CR^{2-\frac{9}{p}}\|u\|_{L^p(A_R)}^3=\lim_{R\rightarrow+\infty} CR^{2-\frac{9}{p}+3\alpha}[X_{p,\alpha}(R)]^3=0\notag,
\end{align*}
there exists a constant $R_7>R_3$ such that
\begin{align}\label{ine4.39}
 CR^{2-\frac{9}{p}}\|u\|_{L^p(A_R)}^3\leq\frac{1}{8}E(R_3)\leq\frac{1}{8}E(R),\;\forall R\geq R_7.
\end{align}
Putting \eqref{ine4.30}, \eqref{ine4.31} and \eqref{ine4.39} into \eqref{ine4.38}, we obtain
\begin{align*}
E(R)\leq& C\left[R\ln R E'(R)\right]^{\theta}+\frac{1}{2}E(R),
\end{align*}
which gives
\begin{align*}
E(R)
&\leq C\left[R\ln RE'(R)\right]^{\theta}.
\end{align*}
Routinely, we have $u=b=\omega=0$.
\end{proof}

The proof of Theorem \ref{main4} is similar to that of Theorem \ref{main2}.
Here we provide another slightly different method to establish the energy estimate. We briefly illustrate our idea and only consider several key terms that we adopt different techniques to handle. If we do not replace $-\Delta U$ by $\curl^2 U$ in \eqref{equ4.3}, then we will get the following equality
\begin{align*}
\int_{B_{2R}}(|\nabla U|^2 +\kappa|\omega|^2)\eta dx+\cdots=\int_{B_{2R}}\curl U\cdot \omega \eta dx+\cdots.
\end{align*}
Applying the Young inequality to the above equality, we obtain
\begin{align}\label{ine4.42}
\aligned
\int_{B_{2R}}(|\nabla U|^2 +\kappa|\omega|^2)\eta dx+\cdots&\leq \int_{B_{2R}}|\curl U|\eta^\frac{1}{2}\cdot |\omega| \eta^\frac{1}{2}dx+\cdots\\
&\leq\int_{B_{2R}}\left(\varepsilon|\curl U|^2\eta+ \frac{1}{4\varepsilon}|\omega|^2\eta\right)dx+\cdots.
\endaligned
\end{align}
Using \eqref{ine4.16}, we can deduce
\begin{align}\label{ine4.43}
\int_{B_{2R}}|\curl  U|^2\eta dx\leq&\int_{B_{2R}}|\nabla U|^2\eta dx+C\|\nabla U\|_{L^2(A_R)}^2.
\end{align}
Combining \eqref{ine4.42} and \eqref{ine4.43}, we infer
\begin{align*}
\int_{B_{2R}}\left[(1-\varepsilon)|\nabla U|^2 +\left(\kappa-\frac{1}{4\varepsilon}\right)|\omega|^2\right]\eta dx+\cdots&\leq C\|\nabla U\|_{L^2(A_R)}^2+\cdots.
\end{align*}

\subsection*{Acknowledgements}
This work was supported by Science Foundation for the Excellent Youth Scholars of Higher Education of Anhui Province (Grant No. 2023AH030073), and Domestic Study and Research Support Program for Young Key Teachers of Higher Education of Anhui Province (Grant No. JNFX2025027).

\subsection*{Data Availability Statement}
No data was used for the research described in the article.

\subsection*{Conflict of Interest Statement}
The authors have no relevant financial or non-financial interests to disclose.


\vspace {0.1cm}

\begin {thebibliography}{DUMA}
\bibitem{AS74}  G. Ahmadi, M. Shahinpoor, Universal stability of magneto-micropolar fluid motions, Internat. J. Engrg. Sci. 12(7) (1974)
657-663.

\bibitem{BRF03} J. Boldrini, M. Rojas-Medar, E. Fern\'{a}ndez-Cara, Semi-Galerkin approximation and strong solutions to the equations of the nonhomogeneous asymmetric fluids, J. Math. Pures Appl. (9) 82 (2003), no. 11, 1499-1525.


\bibitem{Chae14} D. Chae, Liouville-type theorems for the forced Euler equations and the Navier-Stokes equations, Comm. Math. Phys. 326(1) (2014) 37-48.

\bibitem{CL24} D. Chae, J. Lee, On Liouville type results for the stationary MHD in $\mathbb{R}^3$, Nonlinearity 37(9) (2024), Paper No. 095006, 15 pp.

\bibitem{CW19} D. Chae, J. Wolf, On Liouville type theorem for the stationary Navier-Stokes equations, Calc. Var. Partial Differential Equations 58(3) (2019), Paper No. 111, 11 pp.

\bibitem{DJL21} D. Chamorro, O. Jarr\'{\i}n, P.-G. Lemari\'{e}-Rieusset, Some Liouville theorems for stationary Navier-Stokes equations in Lebesgue and Morrey spaces, Ann. Inst. H. Poincar\'{e} C Anal. Non Lin\'{e}aire, 38(3) (2021) 689-710.

\bibitem{CLV} D. Chamorro, D. Llerena, G. Vergara-Hermosilla, Some remarks about the stationary micropolar fluid equations: existence, regularity and uniqueness, J. Math. Anal. Appl. 536(2) (2024), Paper No. 128201, 23 pp.

\bibitem{CNY24} Y. Cho, J. Neustupa, M. Yang, New Liouville type theorems for the stationary Navier-Stokes, MHD, and Hall-MHD equations, Nonlinearity 37(3) (2024), Paper No. 035007, 22 pp.

\bibitem{CNY25} Y. Cho, J. Neustupa, M. Yang, New Liouville-type theorems for stationary solutions of the equations of motion of a magneto-micropolar fluid, J. Differential Equations 441 (2025), Paper No. 113488, 22 pp.

\bibitem{CY25} Y. Cho, M. Yang, Logarithmic improvement of a Liouville-type theorem for the stationary Navier-Stokes equations, preprint, arXiv:2501.04372.

\bibitem{Galdi} G.P. Galdi, An introduction to the Mathematical Theory of the Navier-Stokes Equations: Steady-State Problems, 2nd edn., Springer Monographs in Mathematics, Springer, New York, 2011.

\bibitem{Giaquinta} M. Giaquinta, Multiple Integrals in the Calculus of Variations and Nonlinear Elliptic Systems, Princeton University Press, Princeton, New Jersey, 1983.

\bibitem{KK} J.M. Kim, S. Ko, Some Liouville-type theorems for the stationary 3D magneto-micropolar fluids, Acta Math. Sci. Ser. B 44(6) (2024) 2296-2306.

\bibitem{KNSS09} G. Koch, N. Nadirashvili,  G. Seregin, V. $\check{S}$ver$\acute{a}$k, Liouville theorems for the Navier-Stokes equations and applications, Acta Math. 203(1) (2009) 83-105.

\bibitem{KTW17} H. Kozono, Y. Terasawa, Y. Wakasugi,  A remark on Liouville-type theorems for the stationary Navier-Stokes equations in three space dimensions, J. Funct. Anal. 272(2) (2017) 804-818.

\bibitem{LLZ24} X. Lin, C. Liu, T. Zhang, Magneto-micropolar boundary layers theory in Sobolev spaces without monotonicity: well-posedness and convergence theory, Calc. Var. Partial Differential Equations 63(3) (2024), Paper No. 76, 62 pp.

\bibitem{R97} M. Rojas-Medar, Magneto-micropolar fluid motion: existence and uniqueness of strong solution, Math. Nachr. 188 (1997) 301-319.

\bibitem{RB88} M. Rojas-Medar, J. Boldrini, Magneto-micropolar fluid motion: existence of weak solutions, Rev. Mat. Complut. 11(2) (1988) 443-460.

\bibitem{Seregin16} G. Seregin, Liouville type theorem for stationary Navier-Stokes equations, Nonlinearity 29(8) (2016) 2191-2195.

\bibitem{Seregin18} G. Seregin, Remarks on Liouville type theorems for steady-state Navier-Stokes equations, Algebra i Analiz   30(2) (2018) 238-248;  reprinted in  St. Petersburg Math. J.  30(2)  (2019) 321-328.

\bibitem{SW19} G. Seregin, W. Wang, Sufficient conditions on Liouville type theorems for the 3D steady Navier-Stokes equations, Algebra i Analiz 31(2) (2019) 269-278; reprinted in St. Petersburg Math. J. 31(2) (2020) 387-393.

\bibitem{TWZ19} Z. Tan, W. Wu, J. Zhou, Global existence and decay estimate of solutions to magneto-micropolar fluid equations, J. Differ. Equ. 266(7) (2019) 4137-4169.

\bibitem{Tsai21} T.P. Tsai, Liouville type theorems for stationary Navier-Stokes equations, Partial Differ. Equ. Appl. 2(1) (2021), Paper No. 10, 20 pp.

\bibitem{ZWX25} X. Zhai, J. Wu, F. Xu, Stability for the 3D magneto-micropolar fluids with only velocity dissipation near a background magnetic field, J. Differ. Equ. 425 (2025) 596-626.

\bibitem{ZZB25} Z. Zhang, New Liouville type theorems for 3D steady incompressible MHD equations and Hall-MHD equations, preprint, arXiv:2503.13202. 	

\bibitem{Zhao23} J. Zhao, Mild ill-posedness for 2D magneto-micropolar fluid equations near a background magnetic field, SIAM J. Math. Anal. 55 (2023), no. 5, 5967-5992.


\end{thebibliography}
\end {document}